\documentclass[english,russian,12pt]{amsart} 
\usepackage{nicefrac}
\usepackage[inline]{enumitem}
\usepackage{enumitem}
\usepackage{amssymb}
\usepackage[usenames,dvipsnames,svgnames,table]{xcolor}
\usepackage{fouriernc}%%%%%%%%%%%%%% mettre les 3 lignes
\usepackage{courier}
\usepackage[marginparwidth=2cm]{geometry}
\usepackage{array,float}

\usepackage{todonotes}
\usepackage{graphicx,subfigure,epsfig}
\usepackage{subfigure}
\usepackage{epstopdf}
\usepackage{mathtools}

\usepackage{amssymb,amsmath}

\newtheorem{theorem}{Theorem}%[section]
\newtheorem{definition}{Definition}%[section]
\newtheorem{lemma}{Lemma}%[section]
\newtheorem{remark}{Remark}%[section]

\newtheorem{corollary}{Corollary}%[section]
\newtheorem{example}{Example}%[section]
\newtheorem{proposition}{Proposition}%[section]

 \newcommand{\ind}{{\operatorname{Ind}}}
 \newcommand{\sgn}{{\operatorname{sgn}}}
 \newcommand{\nJ}{{\nabla J}}
\newcommand{\VL}{{\mathcal{VL}_{\text{flat}}}}

\title{Recurrent Generalization of F-Polynomials for Virtual Knots and Links}

\thanks{A. Gill and M. Prabhakar were supported by DST (project number DST/INT/RUS/RSF/P-02), M. Ivanov was supported by RFBR (grant number 19-01-00569),  A. Vesnin was  supported by RSF (grant number 20-61-46005).}

\author{Amrendra Gill}
\address{Department of Mathematics, Indian Institute of Technology Ropar; 140001 Punjab, India}
\email{amrendra.gill@iitrpr.ac.in}

\author{Maxim Ivanov}
\address{Laboratory of Topology and Dynamics, Novosibirsk State University, 630090 Novosibirsk, Russia}
\email{m.ivanov2@g.nsu.ru}

\author{Madeti Prabhakar}
\address{Department of Mathematics, Indian Institute of Technology Ropar; 140001 Punjab, India}
\email{prabhakar@iitrpr.ac.in}

\author{Andrei Vesnin}
\address{Tomsk State University, Tomsk, Russia; Higher School of Economics, Moscow, Russia; Sobolev Institute of Mathematics, Siberian Branch of the Russian Academy of Sciences, 630090 Novosibirsk, Russia}
\email{vesnin@math.nsc.ru}

\begin{document}

\begin{abstract}
F-polynomials for virtual knots were defined by Kaur, Prabhakar and Vesnin in 2018 using flat virtual knot invariants. These polynomials  naturally generalize Kauffman's affine index polynomial and use smoothing in classical crossing of a virtual knot diagram. In this paper we introduce weight functions for ordered orientable virtual and flat virtual link. A flat virtual link is an equivalence class of virtual links in respect to a local symmetry changing  type of classical crossing in a diagram. By considering three types of smoothings in classical crossings of a virtual link diagram and suitable weight functions, we provide a recurrent construction for new invariants. 
We demonstrate by  providing explicit examples, that newly defined polynomial invariants are stronger than F-polynomials.
\end{abstract}

% Keywords

\maketitle

\tableofcontents

\section{Introduction}

Theory of virtual knots and links was introduced by Kauffman in~\cite{Ka99} as a generalization of classical knot theory.  It was observed by Kuperberg in~\cite{Ku03} that the study of virtual knots is naturally related to the study of knots and links embedded in 3-manifolds which are thickened surfaces. Following the Kauffman's approach we consider virtual links as equivalence classes of virtual link diagrams up the equivalence relations corresponding to generalized Reidemeister moves. Recall that a virtual knot or link diagram is 4-regular planar graph, where each vertex is indicated as either classical or virtual crossing.  We will consider oriented virtual links, where each  virtual crossing is depicted by placing a small circle at the vertex. In Figure~\ref{fig1}, there are presented two types of classical crossings and the virtual crossing of an oriented virtual link diagram.

Polynomial invariants based on the index values in classical crossings were introduced by Cheng and Gao~\cite{CG13}, known as th writhe polynomial,  and by Kauffman~\cite{Ka13}, known as the affine index polynomial. The connection of the affine index polynomial with the virtual knot cobordism is described in~\cite{Ka18}. For related polynomial invariants and their properties see~\cite{Ka20, Me16, Pe20, Sa16}. Three kinds of invariants of a virtual knot called the first, second and third intersection polynomials were introduced recently in~\cite{Satoh21}. F-polynomials for oriented virtual knots were constructed in~\cite{KPV18} as a generalization of the affine index polynomial. The construction of F-polynomials is based on invariants of flat virtual knots. Flat virtual link is an equivalence class of virtual links in respect to a local symmetry changing  type of classical crossing in a diagram. Two types of classical crossings are  presented in Figure~\ref{fig1}.   \linebreak  F-polynomials were calculated for tabulated virtual knots in~\cite{IV20} and~\cite{VI20}, and successfully used to distinguish some oriented virtual knots in~\cite{GPV20}.  Another approach to construct invariants of flat virtual knots can be based on representation of flat virtual braids by automorphisms of free groups, see, for example,~\cite{BC20}. 

\begin{figure}[h]
\centering
\unitlength=0.6mm
\begin{picture}(0,30)(0,5)
\thicklines
\qbezier(-70,10)(-70,10)(-50,30)
\qbezier(-70,30)(-70,30)(-62,22)
\qbezier(-50,10)(-50,10)(-58,18)
\put(-65,25){\vector(-1,1){5}}
\put(-55,25){\vector(1,1){5}}
\qbezier(10,10)(10,10)(-10,30)
\qbezier(10,30)(10,30)(2,22)
\qbezier(-10,10)(-10,10)(-2,18)
\put(-5,25){\vector(-1,1){5}}
\put(5,25){\vector(1,1){5}}
\qbezier(70,10)(70,10)(50,30)
\qbezier(70,30)(70,30)(50,10)
\put(55,25){\vector(-1,1){5}}
\put(65,25){\vector(1,1){5}}
\put(60,20){\circle{4}}
\end{picture}
\caption{Two classical and one virtual crossings.} \label{fig1}
\end{figure}
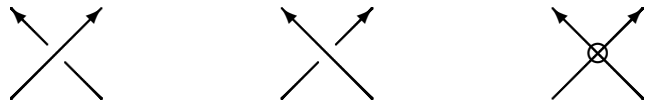

The paper is organized as the following. Section~\ref{sec2}  contains some preliminary information about generalized Reidemeister moves, definitions of sign and index of a classical crossings, orientation revising smoothing of an oriented virtual knot diagram, and $F$-polynomials introduced in~\cite{KPV18}.  In Section~\ref{sec3} we introduce weight functions associated with classical crossings (see Definitions~\ref{def1} and~\ref{def2}). In Section~\ref{sec4} we use weight functions to define I-function and flat I-function (see Definitions~\ref{def5} and~\ref{def6}. We prove in Theorem~\ref{theorem1} that these functions are invariants of ordered oriented virtual links and ordered oriented flat virtual links, respectively.  All introduced  notions are illustrated by Examples~\ref{example1}--\ref{example4}.   In Section~\ref{sec5} we consider two more types of smoothings in classical crossings and construct corresponding weight functions taking values in the free $\mathbb Z$-module, generated by ordered oriented flat virtual links (see Theorem~\ref{theorem2}). Such invariants, corresponding to type-2 smoothings, are constructed in Corollary~\ref{cor1} and used in Example~\ref{example5} to demonstrate that the virtual Kishino knot, a famous connected sum of two trivial virtual knots, is non-trivial. In Section~\ref{sec6} we presents a recurrent construction of a sequences of invariants, see Proposition~\ref{prop1} and realize this method to define a multi-variable generalization of the $F$-polynomial in Theorem~\ref{theorem3}. In Section~\ref{sec7} we define $(n,m)$-difference writhe of a virtual knot diagram and use it to construct an invariant $F_K^{n,m,k} (t, \ell_1, \ell_2)$ of an oriented virtual knots in Theorem~\ref{theorem4}. Also, we demonstrate  that this invariant is stronger than $F$-polynomials. In Section~\ref{sec8} we introduce an ordered virtual link invariant denoted by $\operatorname{span}_{n,k} (L)$ and its flat version $\operatorname{fspan}_{n,k} (L)$, that is an ordered flat virtual link invariant, see Theorem~\ref{theorem5}. Then we use $\operatorname{fspan}_{n,k} (L)$  to construct a family of oriented virtual knot invariants $\widetilde{F}^{n,k,m}(t, \ell, v)$ on variables $t$, $\ell$ and $v$ in Theorem~\ref{theorem6} and demonstrate that these  3-variable polynomials are stronger than F-polynomials.

\section{Preliminaries} \label{sec2}

We will say that $n$-component link $L$ is \emph{ordered}, if its components are labelled by different integers from $1$ to $n$. Analogously, ordered virtual links and ordered flat virtual links can be defined. Ordered knots are the particular 1-component case. 

In this paper we consider ordered oriented virtual links and ordered flat virtual links. Forgetting about ordering, we will get usual oriented virtual links and oriented flat virtual links. When ordering is not important in our considerations, we don't present labels of components in figures.  

%ordered oriented virtual links and oriented virtual knots as their special case. We also study ordered oriented flat virtual links and oriented flat virtual knots. Hereafter we will simply write links, knots, flat links and flat knots to denote them.

Two diagrams of ordered oriented virtual links are \emph{equivalent} if and only if one can be obtained from another by a finite sequence of generalized Reidmeister moves. By generalized Reidemeister moves we mean the union of classical Reidemeister moves and virtual Reidemeister moves. The non-oriented versions of these moves are presented  in Figure~\ref{fig2}, and oriented versions can be obtained by considering all possible orientations of link components. 
 \begin{figure}[ht]
\centering
\subfigure[Classical Reidemeister moves.]
{\includegraphics[height=4.cm]{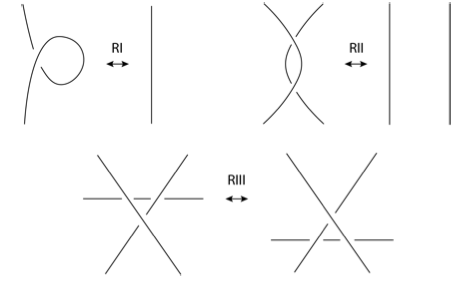} 
} \medskip \subfigure[Virtual Reidemeister moves.]
{
\includegraphics[height=4.cm]{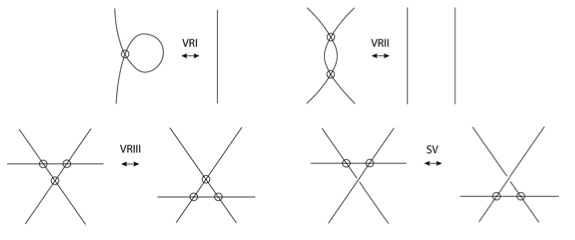} 
}		
\caption{Reidemeister moves.} \label{fig2}
\end{figure}

 An ordered oriented virtual link is defined as an equivalence class of link diagrams modulo generalized Reidemeister moves. For a link diagram $D$ we denote the set of its classical crossings by $C(D)$. By $D^{*}$ we denote a \emph{mirror image} of $D$ that is  a diagram obtained from $D$ by changing all classical crossings, and for $c \in C(D)$ we denote by  $c^{*}$ a crossing corresponding to $c$ in $C(D^{*})$.  An arc of a diagram is an edge of a corresponding 4-regular graph.
An object (or a quantity) associated to a diagram which remains invariant under all generalized Reidemeister moves is called an \emph{ordered oriented virtual link invariant}.

The \emph{sign} of a classical crossing $c \in C(D)$, denoted by $\operatorname{sgn}(c)$, is defined as in Figure~\ref{fig3}.
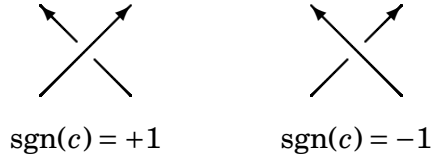
\begin{figure}[!ht]
\centering
\unitlength=0.6mm
\begin{picture}(0,35)
\thicklines
\qbezier(-40,10)(-40,10)(-20,30)
\qbezier(-40,30)(-40,30)(-32,22)
\qbezier(-20,10)(-20,10)(-28,18)
\put(-35,25){\vector(-1,1){5}}
\put(-25,25){\vector(1,1){5}}
\put(-30,0){\makebox(0,0)[cc]{$\operatorname{sgn}(c)=+1$}}
\qbezier(40,10)(40,10)(20,30)
\qbezier(40,30)(40,30)(32,22)
\qbezier(20,10)(20,10)(28,18)
\put(25,25){\vector(-1,1){5}}
\put(35,25){\vector(1,1){5}}
\put(30,0){\makebox(0,0)[cc]{$\operatorname{sgn}(c)=-1$}}
\end{picture}
\caption{Crossing signs.} \label{fig3}
\end{figure}

Next we assign an integer value to each arc in $D$ in such a way that the labeling around each crossing point of $D$ follows the rule as shown in Figure~\ref{fig4}.
\begin{figure}[!ht]
\centering
\unitlength=0.6mm
\begin{picture}(0,35)(0,5)
\thicklines
\qbezier(-70,10)(-70,10)(-50,30)
\qbezier(-70,30)(-70,30)(-62,22)
\qbezier(-50,10)(-50,10)(-58,18)
\put(-65,25){\vector(-1,1){5}}
\put(-55,25){\vector(1,1){5}}
\put(-75,34){\makebox(0,0)[cc]{$b+1$}}
\put(-75,8){\makebox(0,0)[cc]{$a$}}
\put(-45,8){\makebox(0,0)[cc]{$b$}}
\put(-45,34){\makebox(0,0)[cc]{$a-1$}}
\qbezier(10,10)(10,10)(-10,30)
\qbezier(10,30)(10,30)(2,22)
\qbezier(-10,10)(-10,10)(-2,18)
\put(-5,25){\vector(-1,1){5}}
\put(5,25){\vector(1,1){5}}
\put(-15,34){\makebox(0,0)[cc]{$b+1$}}
\put(-15,8){\makebox(0,0)[cc]{$a$}}
\put(15,8){\makebox(0,0)[cc]{$b$}}
\put(15,34){\makebox(0,0)[cc]{$a-1$}}
\qbezier(70,10)(70,10)(50,30)
\qbezier(70,30)(70,30)(50,10)
\put(55,25){\vector(-1,1){5}}
\put(65,25){\vector(1,1){5}}
\put(60,20){\circle{4}}
\put(45,34){\makebox(0,0)[cc]{$b$}}
\put(45,8){\makebox(0,0)[cc]{$a$}}
\put(75,8){\makebox(0,0)[cc]{$b$}}
\put(75,34){\makebox(0,0)[cc]{$a$}}
\end{picture}
\caption{Labeling around crossing.} \label{fig4}
\end{figure}
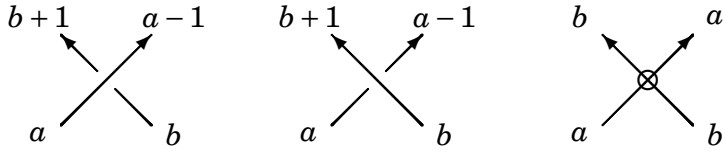
The \emph{index} value for a classical crossing $c \in C(D)$, denoted by $\operatorname{Ind}(c)$, is defined as 
\begin{equation} 
\operatorname{Ind}(c) = \operatorname{sgn} (c)(a-b-1). \label{eqn1}
\end{equation}
Then the \emph{affine index polynomial} of a knot $K$ is defined via its diagram $D$ as
$$
P_{K}(t) = \sum _{c \in C(D)} \operatorname{sgn}(c)(t^{\operatorname{Ind}(c)}-1).
$$
For properties and applications of $P_{K}(t)$ see~\cite{CG13, Ka13, Ka18, Ka20}.

For each $n \in \mathbb{Z} \setminus \{0\}$ the  \emph{$n$-th writhe $J_n(D)$} of a virtual knot diagram $D$ is defined as the number of positive sign crossings minus number of negative sign crossings of $D$ with index value $n$. The $n$-th writhe is a virtual knot invariant. Using $n$-th writhe, a new invariant \emph{$n$-th d-writhe} (difference writhe) of $D$ denoted by $\nabla J_{n}(D)$ was defined in~\cite{KPV18} as 
\begin{equation}
 \nabla J_{n}(D)=J_{n}(D)-J_{-n}(D). \label{eqn2}
\end{equation}
Thus $\nabla J_{n}(D)$ is a flat virtual knot invariant. For a classical crossing  $c \in C(D)$ denote by $D_c$ a diagram, obtained from $D$ by the \emph{type-1 smoothing} of the diagram $D$ at the crossing $c$, see Figure~\ref{fig5}. This smoothing was used in~\cite{KPV18} and called a \emph{smoothing against orientation}. 
\begin{figure}[ht]
	\centering
	\includegraphics[scale=0.6]{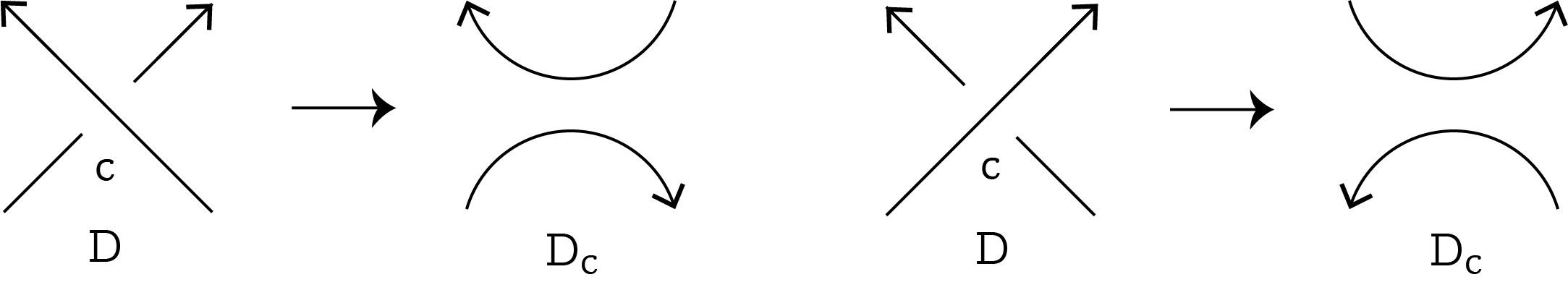}
	\caption{Type-1 smoothing (smoothing against orientation).} \label{fig5}
\end{figure}

Values $\nabla J_n(D_c)$ for $c \in C(D)$ were used in~\cite{KPV18} to construct a family of polynomial invariants $F^n_K(t, \ell)$. An \emph{$n$-th $F$-polynomial} of a knot $K$ is defined via its diagram $D$ as
$$
F_{K}^{n}(t,\ell) = \sum_{c \in C(D)} \operatorname{sgn}(c)t^{\text{Ind}(c)} \ell^{\nabla J_{n}(D_{c})} -  \sum _{c\in T_{n}(D)} \operatorname{sgn}(c) \ell^{\nabla J_{n}(D_{c})} - \sum _{c\notin T_{n}(D)} \operatorname{sgn} (c) \ell^{\nabla J_{n}(D)},
$$
where  $T_n(D)$ consists of classical crossings of the diagram $D$ with the following property:
$$
T_{n}(D)=\{c \in C(D)  \, \mid  \,  \nabla J_{n}(D_{c}) = \pm \nabla J_{n}(D)\}.
$$
For properties and applications of $F_{K}^{n}(t,\ell)$ see~\cite{IV20, KPV18, VI20}.

\section{Weight functions} \label{sec3}

Let $\mathcal D$ be a subset of all ordered oriented virtual link diagrams with the following property: if $D \in \mathcal D$ then all diagrams obtained from $D$ by generalized Reidemeister moves, crossing change operation, reversing orientation and reordering of components also belong to $\mathcal D$. We call $\mathcal D$ a \emph{regular set of diagrams}. The corresponding set of links is said to be a \emph{regular set of ordered oriented virtual links}. By forgetting ordering of components we obtain a \emph{regular set of unordered oriented virtual links}, and similarly, by forgetting types of classical crossings  we obtain a \emph{regular set of ordered oriented flat virtual links}. 

Denote by $C(\mathcal D)$ the set of all classical crossings of diagrams $D \in \mathcal D$.

\begin{definition} \label{def1} \rm 
Let  $G$ be an abelian group and $w : C(\mathcal D) \to G$ be a function which assigns  a \emph{value} $w(c) \in G$ to a classical crossing $c \in C(D)$ for all diagrams $D \in \mathcal D$. Function $w$ is said to be a \emph{weight function}, write $w \in W_G$, if it satisfies \emph{weight function conditions} (C1)--(C3):
\begin{itemize}
\item[(C1)]  $w$ is \emph{local}, i.e. if $D'$ is obtained from $D$ by a generalized Reidemeister move such that a crossing $c \in D$ is not involved in this move and $c' \in D'$ is the corresponding crossing, then $w(c') = w(c)$;
   
\item[(C2)] if diagram $D'$ is obtained from $D$ by RIII-move and involved classical crossings $a, b, c \in D$ have weights $w(a)$, $w(b)$ and $w(c)$, as well as involved crossings of $a', b', c' \in D'$ have weights $w(a')$, $w(b')$ and $w(c')$, see Figure~\ref{fig6}, then $w(a')= w(a)$, $w(b') = w(b)$ and $w(c') =w(c)$.
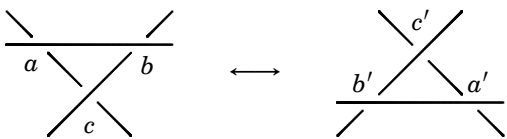
\begin{figure}[!ht]
	\centering
	\unitlength=0.55mm
	\begin{picture}(0,35)(0,5)
		\thicklines
		\qbezier(-50,10)(-50,10)(-30,30)
		\qbezier(-50,30)(-50,30)(-42,22)
		\qbezier(-30,10)(-30,10)(-38,18)
		\qbezier(-60,32)(-60,32)(-20,32)
		\qbezier(-54,34)(-54,34)(-60,40)
		\qbezier(-26,34)(-26,34)(-20,40)
		\put(-54,27){\makebox(0,0)[cc]{\footnotesize $a$}}
		\put(-26,27){\makebox(0,0)[cc]{\footnotesize $b$}}
		\put(-40,12){\makebox(0,0)[cc]{\footnotesize $c$}}
		\put(0,25){\makebox(0,0)[cc]{$\longleftrightarrow$}}
		\qbezier(50,40)(50,40)(30,20)
		\qbezier(50,20)(50,20)(42,28)
		\qbezier(30,40)(30,40)(38,32)
		\qbezier(20,18)(20,18)(60,18)
		\qbezier(20,10)(20,10)(26,16)
		\qbezier(60,10)(60,10)(54,16)
		\put(54,23){\makebox(0,0)[cc]{\footnotesize $a'$}}
		\put(26,23){\makebox(0,0)[cc]{\footnotesize $b'$}}
		\put(40,38){\makebox(0,0)[cc]{\footnotesize $c'$}}
	\end{picture}
	\caption{RIII move.} \label{fig6}
\end{figure}
\item[(C3)] if diagram $D'$ is obtained from $D$ by SV-move and involved classical crossing $c \in D$ has weight $w(c)$, as well as involved classical crossing $c' \in D'$ has weight $w'(c')$, see  Figure~\ref{fig7}, then $w'(c') = w(c)$.
\begin{figure}[!ht]
	\centering
	\unitlength=0.55mm
	\begin{picture}(0,45)(0,0)
		\thicklines
		\qbezier(-50,10)(-50,10)(-42,18)
		\qbezier(-20,40)(-20,40)(-38,22)
		\qbezier(-60,40)(-60,40)(-30,10)
		\qbezier(-60,30)(-60,30)(-20,30)
		\put(-50,30){\circle{4}}
		\put(-30,30){\circle{4}}
		\put(-40,12){\makebox(0,0)[cc]{\footnotesize $c$}}
		\put(0,25){\makebox(0,0)[cc]{$\longleftrightarrow$}}
		\qbezier(20,20)(20,20)(60,20)
		\qbezier(30,40)(30,40)(60,10)
		\qbezier(20,10)(20,10)(38,28)
		\qbezier(50,40)(50,40)(42,32)
		\put(30,20){\circle{4}}
		\put(50,20){\circle{4}}
		\put(40,38){\makebox(0,0)[cc]{\footnotesize $c'$}}
	\end{picture}
	\caption{SV move.} \label{fig7}
\end{figure}
\end{itemize}
\end{definition}

Definition~\ref{def1} may be thought of as a generalization of the Chord Index axioms in~\cite{C16}.

\begin{definition} \label{def2}  \rm 
Let $w : C(\mathcal D) \to G$ be a weight function. Assume that  diagram $D'$ is obtained from $D$ by RII-move and $\alpha$, $\beta$ are crossings involved. 
If  $w(\beta) = - w(\alpha)$,  then we say  that \emph{$w$ is an odd weight function} and write $w \in W^{odd}_G$. If $w(\beta)  = w(\alpha)$, then  we say that \emph{$w$ is an even weight function} and write $w \in W^{even}_G$.
\end{definition}

\begin{example}  \label{example1} \rm 
Let $\alpha$ and $\beta$ be classical crossings involved in RII move as in  Figure~\ref{fig8}. 
\begin{figure}[!ht]
\centering
\unitlength=0.4mm
	\centering
	\includegraphics[scale=0.6]{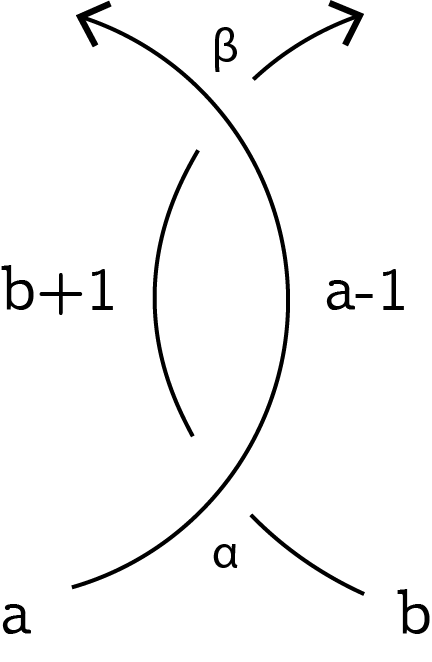}
	\caption{RII move involving crossings $\alpha$ and $\beta$.} \label{fig8}
\end{figure}
Consider two functions $w_1, w_2 : C(D) \to \mathbb Z$, where $w_{1} (c) = \operatorname{sgn} (c)$, is the sign of crossing, defined in Figure~\ref{fig3}, and  $w_{2} (c) =  \operatorname{Ind} (c)$, is the index of crossing, defined by (\ref{eqn1}). Both of them are weight functions. Since  $w_{1} (\alpha) = 1$ and $w_{1}(\beta) = -1$, we get  $w_{1} (\beta) =-  w_{1} (\alpha)$ and hence $w_1 \in W^{odd}_{\mathbb Z}$. Since  $w_{2} (\alpha) = \operatorname{sgn} (\alpha) \cdot (a-b-1) = a-b-1$ and $w_{2} (\beta) = \operatorname{sgn} (\beta) \cdot  (b+1 - (a-1) -1) = a - b -1$, we get  $w_{2} (\beta) = w_{2} (\alpha)$ and hence $w_2 \in W^{even}_{\mathbb Z}$.
\end{example}

\smallskip
For two weight function $u, v : C(D) \to G$, where $G$ is an abelian group, we can define a sum and a product (if the codomain $G$ is a ring) as follows:
\begin{equation}
(u + v)(c) =u(c) + v(c) \qquad \text{and} \qquad  (u*v) (c) = u(c)v(c). \label{eqn3}
\end{equation} 
Therefore $W_G$ is an abelian group with $W_G^{odd}$ and $W_G^{even}$ as subgroups. The set $W_{\mathbb Z}$ with operations (\ref{eqn3}) forms a ring, and it may be convenient to regard $W_G$ as a $W_{\mathbb Z}$-module with the module multiplication denoted by the same symbol "$*$". 

\smallskip 
For any $w \in W_{G_1}$ and any homomorphism $\phi: G_1 \to G_2$ of abelian groups the composition $w' = \phi \circ w$ is a weight  $w' \in  W_{G_2}$.

\smallskip 
We also admit cases when weight functions may be not defined for some crossings of a diagram. In these cases we will use the following approach.  

\begin{definition}  \label{def3} \rm 
A subset $C'(\mathcal{D}) \subset C(\mathcal{D})$ is said to be \emph{consistent} if the characteristic function
$1_{C'(\mathcal{D})}: C(\mathcal{D}) \to \{0, 1\} \subset \mathbb{Z}$ of the set $C'(\mathcal{D})$ is an even weight function.	
\end{definition} 

\begin{example} \label{example2} \rm 
For positive integer $i$ and $j$ consider the regular set $\mathcal DL_n$ of all diagrams of ordered oriented virtual links with at least $n$-components, where $n = \max\{i, j\}$.  Let  $C_i(\mathcal DL_n)$  be the set containing crossings only of the $i$-th component, and let  $C_{ij}(\mathcal DL_n)$ be the set containing only  crossings, that belong to both $i$-th and $j$-th components. The characteristic functions of sets $C_i(\mathcal DL_n)$ and $C_{ij}(\mathcal DL_n)$ are even weight functions.
\end{example}
 
Weight functions for consistent subsets of $C(\mathcal D)$ can be defined in the following way.

\begin{definition} \label{def4} \rm 
Let $C'(\mathcal D) \subset C(\mathcal D)$ be  consistent. Then $w' : C'(\mathcal D) \to G$ is said to be \emph{a weight function defined for $C'(\mathcal D)$} if $w'$ satisfies weight function conditions (C1)~-~(C3) for all crossings in $C'(\mathcal D)$.
\end{definition}

\begin{remark}\label{remark1} \rm 
If $C'(\mathcal D) \subset C(\mathcal D)$ is consistent, and $w' : C'(\mathcal D) \to G$ is a weight function, then $w'$ can be extended  to $w : C(\mathcal D) \to G$ by defining 
$$
w(c) = \begin{cases}
	w'(c), & c \in C'(\mathcal D), \\
	0, & \text{otherwise}.
	\end{cases}
$$
\end{remark}

\section{I-functions} \label{sec4}

Let $C'(\mathcal D)$ and $C''(\mathcal D)$ be consistent subsets of $C(\mathcal D)$.
Suppose there are weight functions $v: C'(\mathcal D) \to G_1 \text{ and } w : C''(\mathcal D) \to G_2$ are such that $v \in W^{odd}_{G_1}$  and {$w \in W^{even}_{G_2}$}. We can assume, that $C''(\mathcal D)$ is a subset of $C'(\mathcal{D})$. Otherwise we can replace $v$ by its extension on $C(\mathcal{D})$ as in Remark~\ref{remark1} and take $C'(\mathcal D) = C(\mathcal D)$.

Let $\mathcal D(L) \subset \mathcal{D}$ be the set of all regular diagrams of an ordered oriented virtual link $L$, and $w : C'(\mathcal D) \to G$ be a weight function. For a diagram $D \in \mathcal D(L)$ denote by $R(w, D)$ the set of weights $w(c)$, where $c$ is a classical crossing in $D$ that may be reduced by a single RI-move.  Then take a union over all diagrams of $L$: $$R(w,L) = \bigcup\limits_{D \in \mathcal D(L)} R(w,D).$$

For  some weight functions the set $R(w,L)$ can be easy described as in Example~\ref{example3}. 

\begin{example}  \label{example3} \rm
(i) Consider the weight function $w(c) = \operatorname{Ind} (c)$. If $c \in C(D)$ can be reduced by  RI-move, then by Figure~\ref{fig4}, the labelling around $c$ is such that  $a = b-1$, hence by (\ref{eqn1}), $\operatorname{Ind} (c) = 0$. Therefore, for any oriented virtual knot $K$ we get $R(\operatorname{Ind}, K) = \{ 0 \}$.\\
(ii) Consider the weight function $w(c) = \nabla J_n(D_c)$. As above, if  $c \in C(D)$ can be reduced by RI-move, then $\operatorname{Ind} (c) = 0$, whence either $\nabla J_n(D_c) = \nabla J_n(D)$ or $\nabla J_n(D_c) = -\nabla J_n(D)$.  Therefore, for any oriented virtual knot $K$ we get $R(w, K) = \{ \pm \nabla J_n(K) \}$.
\end{example}

\begin{definition} \label{def5}  \rm 
Let $D$ be a diagram of an ordered oriented virtual link $L$ and in above notations $g \in G_2$ be such that $g \not\in R(w, L)$ or $R(v, L) = \{0\}$. Define  $I$-function $I : C'( D) \to G_1$ by 
\begin{equation} 
I(D; v, w, g) = \sum_{w(c) = g } v(c). \label{eqn4}
\end{equation} 
\end{definition}

Hereafter we will require that all our weight functions are local with respect to crossing change operation. Then to every weight function $w : C'(\mathcal{D}) \to G$ one can associate a weight function $w^{*} : C'(\mathcal{D}) \to G$
induced by taking a mirror image, i.e. $w^{*} (c) = w (c^{*})$.

\begin{definition} \label{def6} \rm 
Let $D$ be a diagram of an ordered oriented virtual link $L$ and in above notations $g \in G_2$ is such that $g \not\in R(w, L)$ or $R(v, L) = \{0\}$. Define a  flat $I$-function $I _f: C'(D) \to G_1$ by 
\begin{equation}
I_f(D; v, w, g) = \sum_{w(c) = g} v(c) + \sum_{ w^*(c) = g}  v^{*}(c). \label{eqn5} 
\end{equation}
\end{definition}

\begin{theorem} \label{theorem1} In above notations,
	\begin{enumerate}
		\item[(i)] $I(D; v, w, g)$ is an ordered oriented virtual link invariant,
		\item[(ii)] $I_f(D; v, w, g) $ is an ordered oriented  flat virtual link invariant.
	\end{enumerate}
\end{theorem}

\begin{proof}
	(i) Since $I(D; v, w, g)$ is a sum over classical crossings, it is invariant under moves VRI, VRII and VRIII which involve  virtual crossings only.  Moreover, it is invariant under RI-move since the assumption $g \not\in R(w, L)$ or  $R(v, L) = \{0\}$ implies that any crossing involved in RI-move does not participate in the sum.  The invariance under RII-move follows from the assumption $v \in W^{odd}$. Indeed, if two crossings $c_{1}$ and $c_{2}$ in the sum are involved in RII-move, then $v(c_{1}) = - v (c_{2})$.  Invariance under moves RIII and SV  follows from assumptions  (C2) and (C3) of Definition~\ref{def1}.
	
	(ii) Note that $v \in W^{odd}_{G_1}$ implies $v^{*} \in W^{odd}_{G_1}$, and $w \in W^{even}_{G_2}$ implies $w^{*} \in W^{even}_{G_2}$. Therefore, by (i),  $I(D; v^{*}, w^{*},  g)$ is also a virtual link invariant, and hence $I_{f}(D; v, w,  g)$ is a virtual link invariant being a sum of invariants.
	
	To prove that $I_{f} (D; v, w, g)$  is an invariant of flat virtual links, we need to show, that it is invariant under crossing change. Let $D$ be a diagram of $L$ and $c_{0} \in C(D)$. Let $D'$ be a diagram obtained from $D$ by a crossing change in $c_0$, and $c_0'$ be a crossing in $D'$  corresponding to $c_0$. Then we can write
\begin{eqnarray} 
	I_{f}(D; v, w, g) = \sum_{w(c) = g, \, c \neq c_0} v(c) + \sum_{w^*(c) = g, \, c \neq c_0}  v^{*}(c) + S, \label{eqn6}
\end{eqnarray}
	where
	$$
	S = \begin{cases}
	0, & \text{if} \quad w(c_0) \neq g \text{ and } w^*(c_0) \neq g,\\
	v(c_0), & \text{if} \quad w(c_0) = g \text{ and } w^*(c_0) \neq g,\\
	v^*(c_0), & \text{if} \quad w(c_0) \neq g \text{ and }w^*(c_0) = g,\\
	v(c_0) + v^*(c_0), & \text{if} \quad w(c_0) = w^*(c_0) = g.
	\end{cases}
	$$
	Analogously,
\begin{equation}
	I_{f}(D'; v, w,  g) = \sum_{w(c') = g, \,  c' \neq c_0'} v(c') + \sum_{ w^*(c') = g, \, c' \neq c_0' }  v^{*}(c') + S', \label{eqn7}
\end{equation}
	where
	$$
	S' = \begin{cases}
	0, & \text{if} \quad  w(c_0') \neq g \text{ and } w^*(c_0') \neq g,\\
	v(c_0'), & \text{if} \quad  w(c_0') = g \text{ and } w^*(c_0') \neq g,\\
	v^*(c_0'), & \text{if} \quad  w(c_0') \neq g \text{ and }w^*(c_0') = g,\\
	v(c_0') + v^*(c_0'), & \text{if} \quad  w(c_0') = w^*(c_0') = g.
	\end{cases}
	$$
	Since weight functions $v$ and $w$ are local by (C1) of Definition~\ref{def1}, then after crossing change  in $c_0$, in expressions   (\ref{eqn6}) and (\ref{eqn7}) only $S$ and $S'$ may differ. Now let us change all crossings in $D'$ except $c_0'$. Then we obtain a mirror image $D^{*}$ of a diagram $D$. Denote by $c^{*}_{0}$ the crossing in $D^{*}$ which corresponds to $c_{0}$  Hence, $v(c_0') =v(c_0^*)$ and $w(c_0') =w(c_0^*)$. Since by definition $v^*(c) = v(c^*)$ and $w^*(c) = w(c^*)$, we conclude $v(c_0') = v^*(c_0)$ and $w(c_0') = w^*(c_0)$. Similarly we acquire $v^*(c_0') = v^*(c_0^*) = v(c_0)$ and $w^*(c_0') = w^*(c_0^*) = w(c_0)$. Applying those equalities to $S'$ we get imply $S' = S$. Therefore, (\ref{eqn7}) is equal to (\ref{eqn6}), hence $I_f$ is a flat virtual link invariant. 
\end{proof}

By Theorem~\ref{theorem1} we can use notations $I(L; v, w,  g)$ and $I_{f}(L; v, w, g)$ instead of $I(D; v, w, g)$ and $I_{f}(D; v, w,  g)$, where $D$ is a diagram of a virtual link $L$.

\begin{example} \label{example4} \rm 
Consider weight functions $v(c) = \sgn(c) \in W^{odd}_{\mathbb Z}$ and $w(c) = \ind(c) \in W^{even}_{\mathbb Z}$. Then
\begin{equation}
I(D; v, w, n) = \sum\limits_{w (c) = n} v(c) = \sum\limits_{\ind(c) = n} \sgn(c )= J_n(D) \label{eqn8}
\end{equation} 
is the defined above the $n$-th writhe number and
\begin{eqnarray}
I_{f}(D; v, w, n) & = & \sum\limits_{w (c) = n} v(c) + \sum\limits_{w^*(c) = n} v^*(c) \nonumber \\ 
& = & J_n(D) +  \sum\limits_{\ind(c) = -n} (-\sgn(c)) \nonumber \\
& = & J_k(D) - J_{-k}(D) \nonumber \\ 
& = & \nabla J_k(D), \label{eqn9} 
\end{eqnarray}
that is  the $n$-th difference writhe number defined by (\ref{eqn2}). 
\end{example}

\section{Smoothings in classical crossings and invariants} \label{sec5}

Above we used the type-1 smoothing, presented in Figure~\ref{fig5}, to construct $F$-polyno\-mials. Applying the type-1 smoothing to a classical crossing $c \in D$, which belongs to a single connected component, we obtain a link diagram with one less classical crossing and the same number of components.

Let us consider another types of smoothings of virtual link diagrams in classical crossings as below. 
The \emph{type-2 smoothing} is presented in Figure~\ref{fig9}.  Consider a diagram of an $n$-component ordered oriented virtual link and assume that in classical crossing $c$ two meeting arcs belong to the same say, $i$-th, component. After the presented smoothing we will obtain a diagram of an ordered $(n+1)$-component link. The smoothing induces the order change: one arc will preserve orientation and the corresponding component will be the $i$-th, but another arc will get reverse orientation and the corresponding component will be the $(n+1)$-th. 
\begin{figure}[ht]
	\centering
	\includegraphics[scale=0.6]{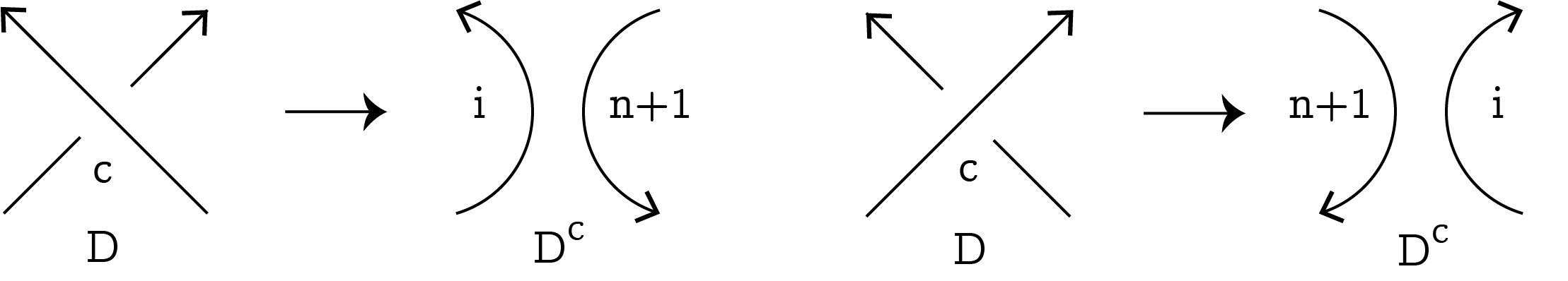}
	\caption{Type-2 smoothing.}.  \label{fig9}
\end{figure}

The \emph{type-3 smoothing} is presented in Figure~\ref{fig10}.  Consider a diagram of 2-component ordered oriented virtual link, and assume that crossing $c$ belong to both components. Then after smoothing we will get a diagram of knot. Note that type-3 smoothing can be generalized to $n$-component ordered oriented virtual links, and we obtain an $(n-1)$-component link as a result.  
\begin{figure}[ht]
	\centering
	\includegraphics[scale=0.6]{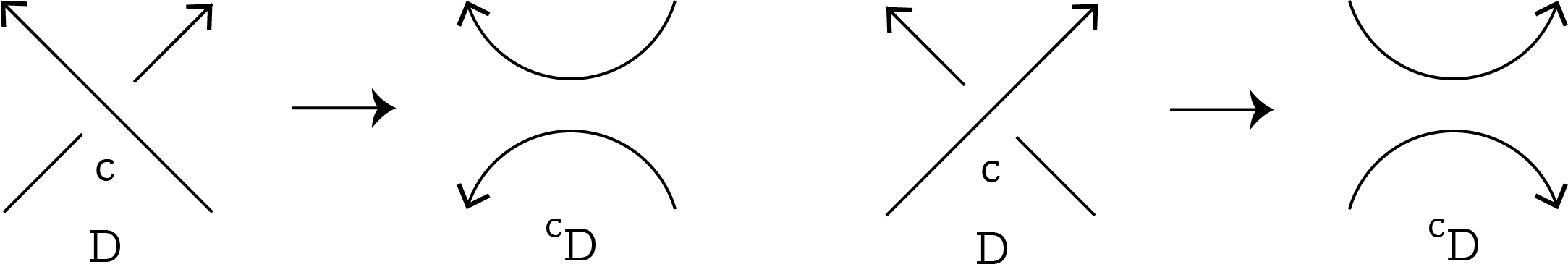}
	\caption{Type-3 smoothing.}
	\label{fig10}
\end{figure}

Further for a classical crossing $c\in D$ we denote by  $D_c$, $D^c$ and $\prescript{c}{}{D}$ a diagram obtained from $D$ by smoothing of type-1, type-2 and type-3, respectively.

\smallskip 
Let us denote by $\VL$ a free $\mathbb Z$-module generated by ordered oriented flat virtual links.
For a virtual link diagram $D$ denote by $[D]$ a flat virtual link whose diagram is obtained from $D$ by replacing all classical crossings by flat crossings. Then $[D] \in \VL$. 
	
\begin{theorem} \label{theorem2}
Functions $S_i : C(\mathcal{D}) \to \VL$, for $i=1,2,3$, defined by  
$$
S_1(c) = [D_c],  \qquad S_2(c) = [D^c], \qquad \text{and} \quad S_3(c) = [\prescript{c}{}{D}], 
$$
are even weight functions. Moreover, if crossing $c \in D$ can be reduced by RI-move, then
\begin{itemize}
\item  $[D_c]$ is equivalent to either $[D]$ or $[D']$, where $D'$ is obtained from $D$ by reversing orientation on the component, containing $c$; 
\item $[D^c]$ is equivalent to $[D]$ with one unknot added, where ordering and orientation of components are induced by type-2 smoothing;  
\item type-3 smoothing can not be applied at crossing $c$. 
\end{itemize}
\end{theorem}

\begin{proof}
The proof is a straightforward check of Reidemeister moves with all possible orientations. For $S_1$ it was done in~\cite[Theorem~3.3]{KPV18}. Now we give a proof for $S_2$. 

\underline{RI-move}: Let $D$ be a diagram of an ordered oriented $n$-component virtual link. Consider crossing $c \in D$ that can be reduced by RI-move on $i$-th component, see Figure~\ref{fig11}. After type-2 smoothing we get a new component which is an unknot. If after  smoothing the $i$-th component remains $i$-th, then its orientation is preserved. If after smoothing the $i$-th component becomes $(n+1)$-th component, then its orientation is reversed. 
\begin{figure}[ht]
\centering
\includegraphics[scale=0.55]{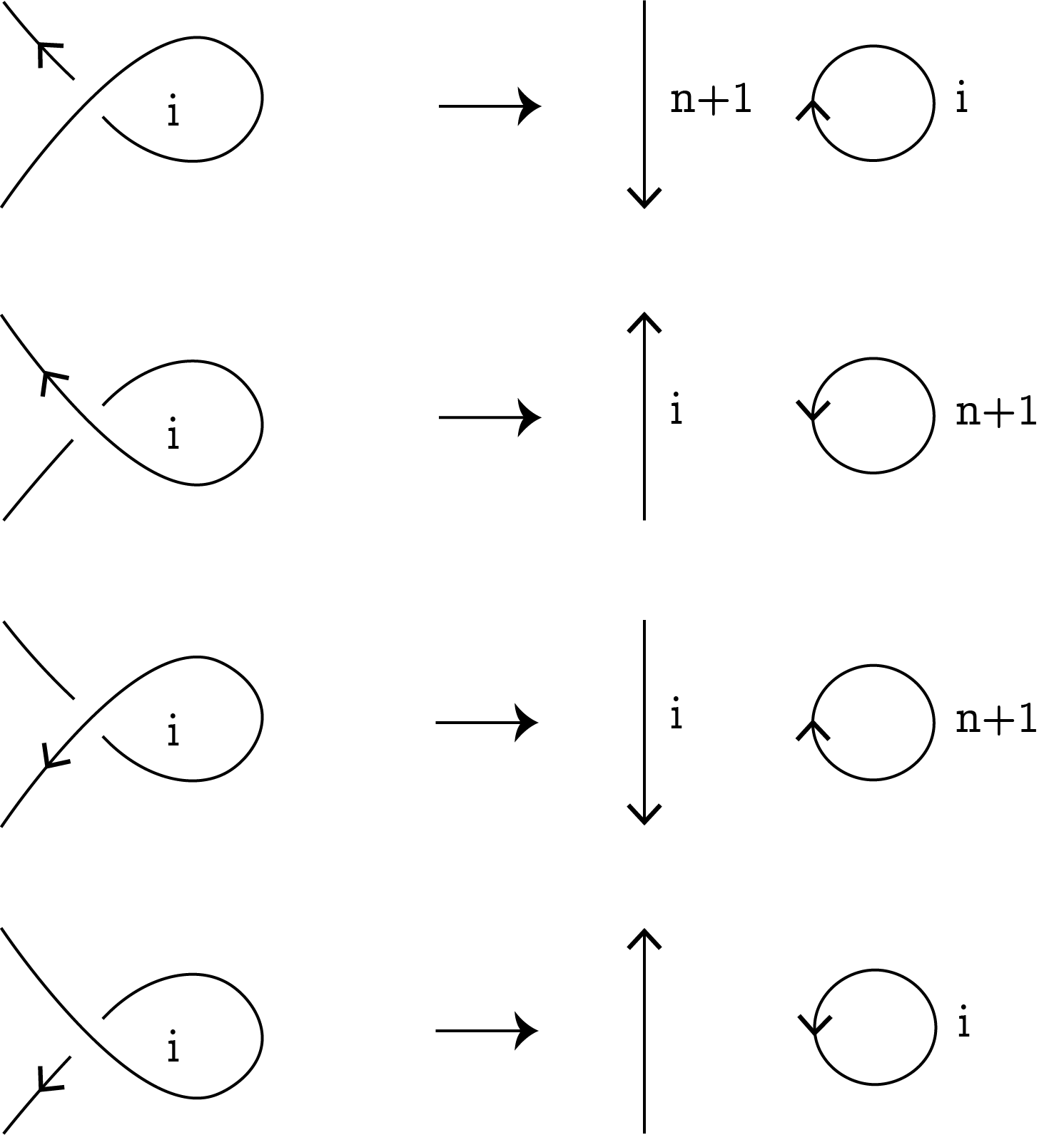}
\caption{All possible RI-moves and corresponding smoothings.} \label{fig11}
\end{figure}

\underline{RII-move}: Let $D$ be a diagram of an ordered oriented $n$-component virtual link. Consider crossings $c_1, c_2 \in D$ that belong to $i$-th component and can be reduced by RII-move. Depending on orientation, there are two cases, presented in Figure~\ref{fig12} and~\ref{fig13}. 
For each case we have two possibilities of type-2 smoothing: since $i$-th component splits into two components, one of them will be either $i$-th  or $(n+1)$-th, and, respectively, another will be $(n+1)$-th or $i$-th.  In the case, presented in Figure~\ref{fig12}, two diagrams, obtained by smoothings, are equivalent under one crossing change, so they are equivalent as flat diagrams. In the case, presented in Figure~\ref{fig13}, two diagrams, obtained by smoothings, are equivalents under RI-moves, so they are equivalent. 
\begin{figure}[ht]
\centering
\includegraphics[scale=0.5]{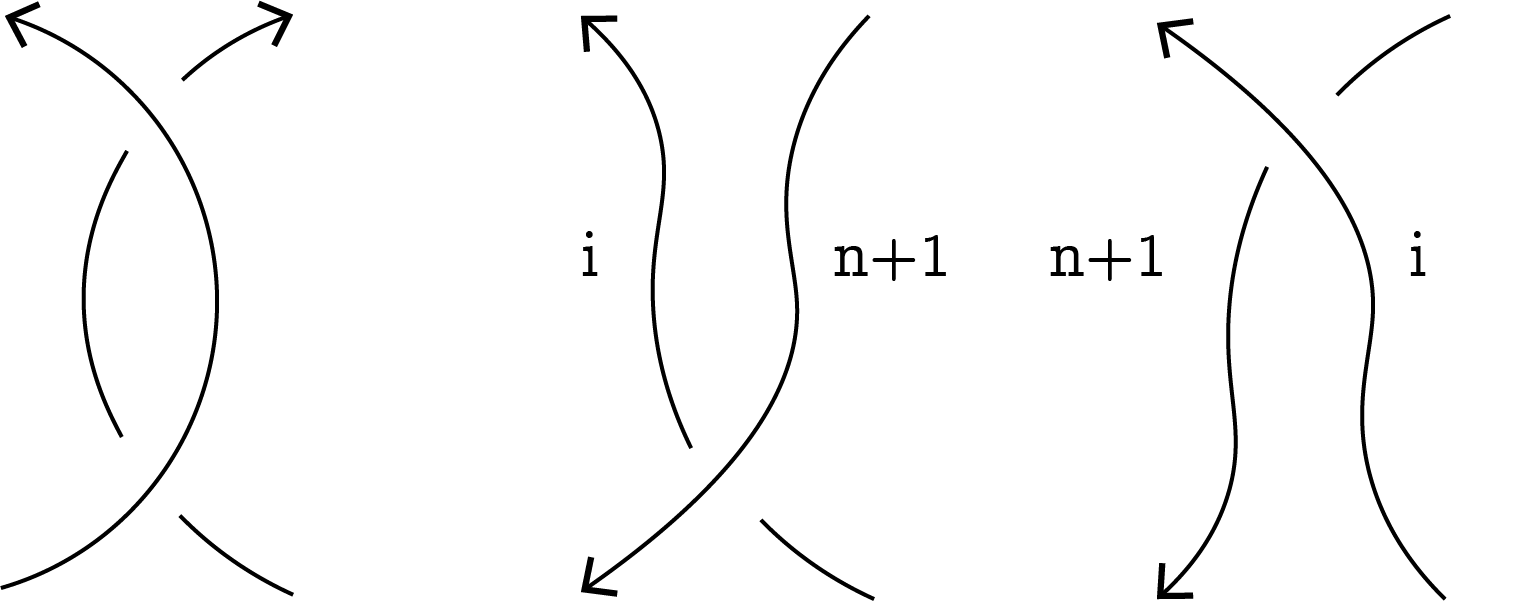}
\caption{1-st case of RII-move and corresponding smoothings.}
\label{fig12}
\end{figure}
\begin{figure}[ht]
\centering
\includegraphics[scale=0.5]{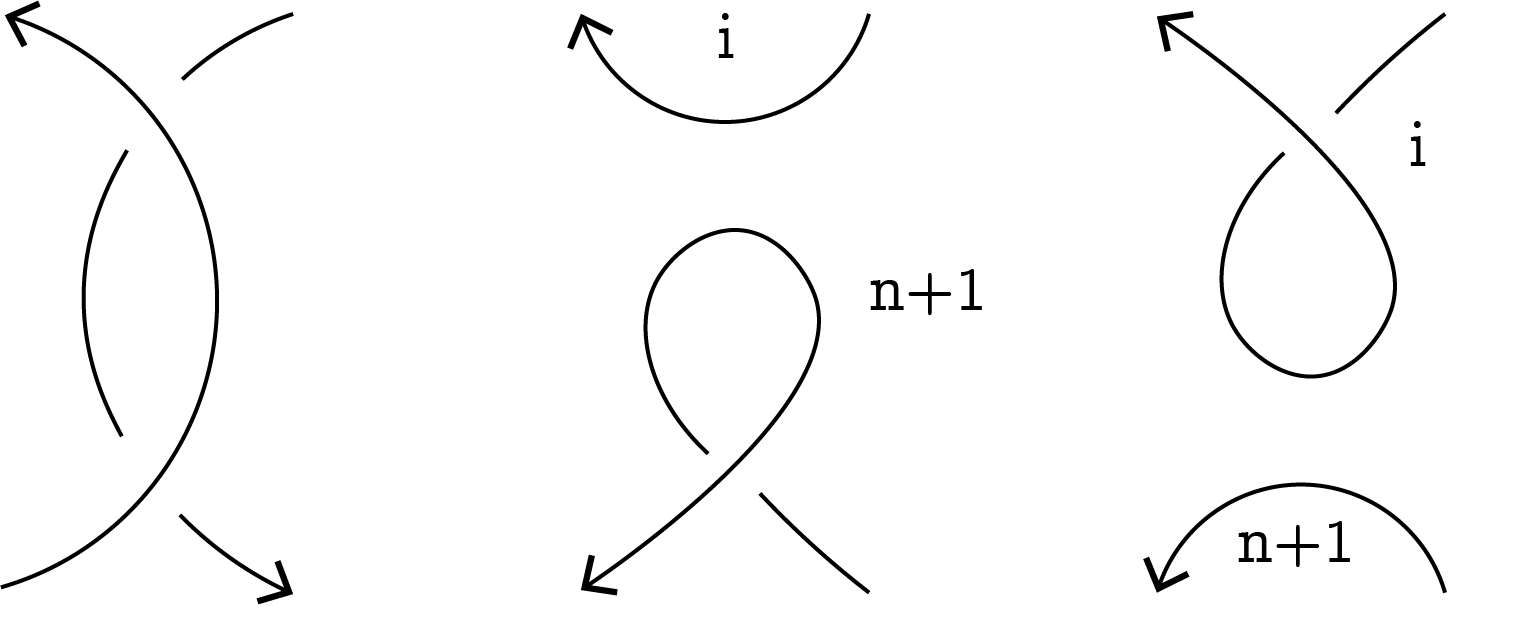}
\caption{2-nd case of RII-move and corresponding smoothings.}
\label{fig13}
\end{figure}

\underline{RIII-move}: Consider RIII-move presented in Figure~\ref{fig14} with components numerated by $i$, $j$ and $k$, where some of these numbers may coincide.  
\begin{figure}[ht]
\centering
\includegraphics[scale=0.5]{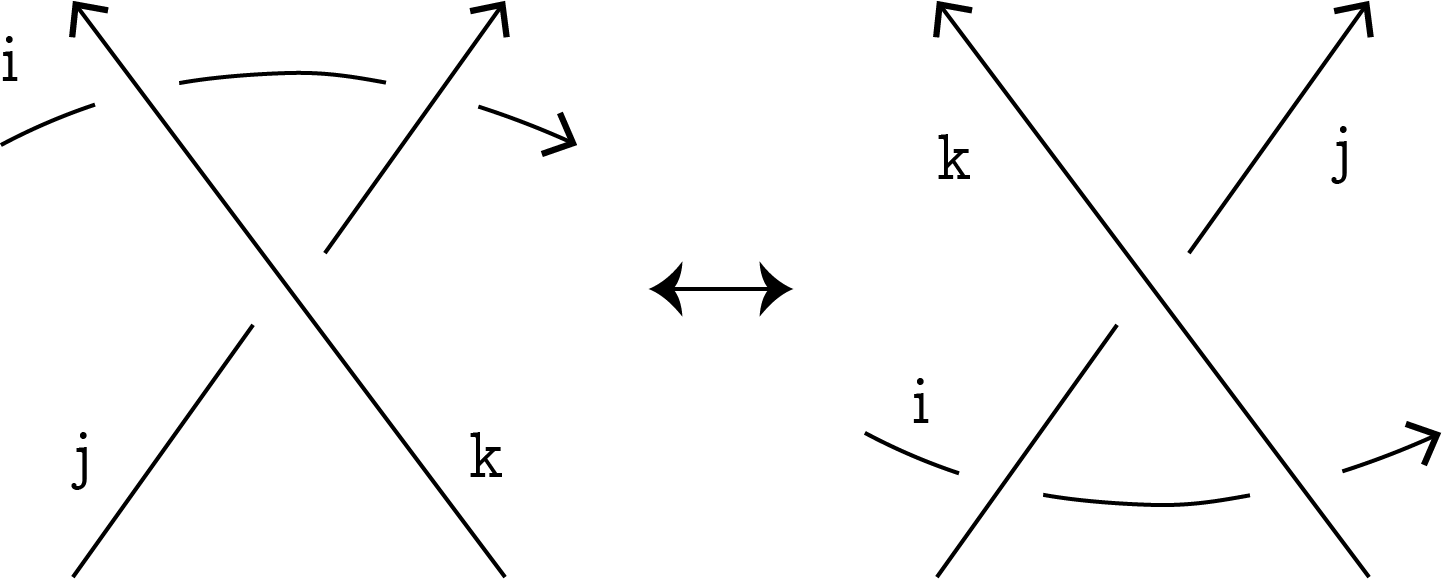}
\caption{RIII-move.} \label{fig14}
\end{figure}
Let us apply type-2 smoothings in three crossings of the initial (before RIII-move) diagram and in three crossings of the terminal (after RIII-move) diagram. Thus, we get three smoothed diagrams of virtual links on the left side in Figure~\ref{fig15} and three smoothed diagrams of virtual links on the right side in the same figure.   It is clear from Figure~\ref{fig15} that smoothed diagrams for corresponding crossing are flat equivalent (for crossing of components $i$ and $k$ ) or  equivalent (for crossing of components $i$ and $j$ and crossing of components $j$ and $k$). 
\begin{figure}[ht]
\centering
\includegraphics[scale=0.5]{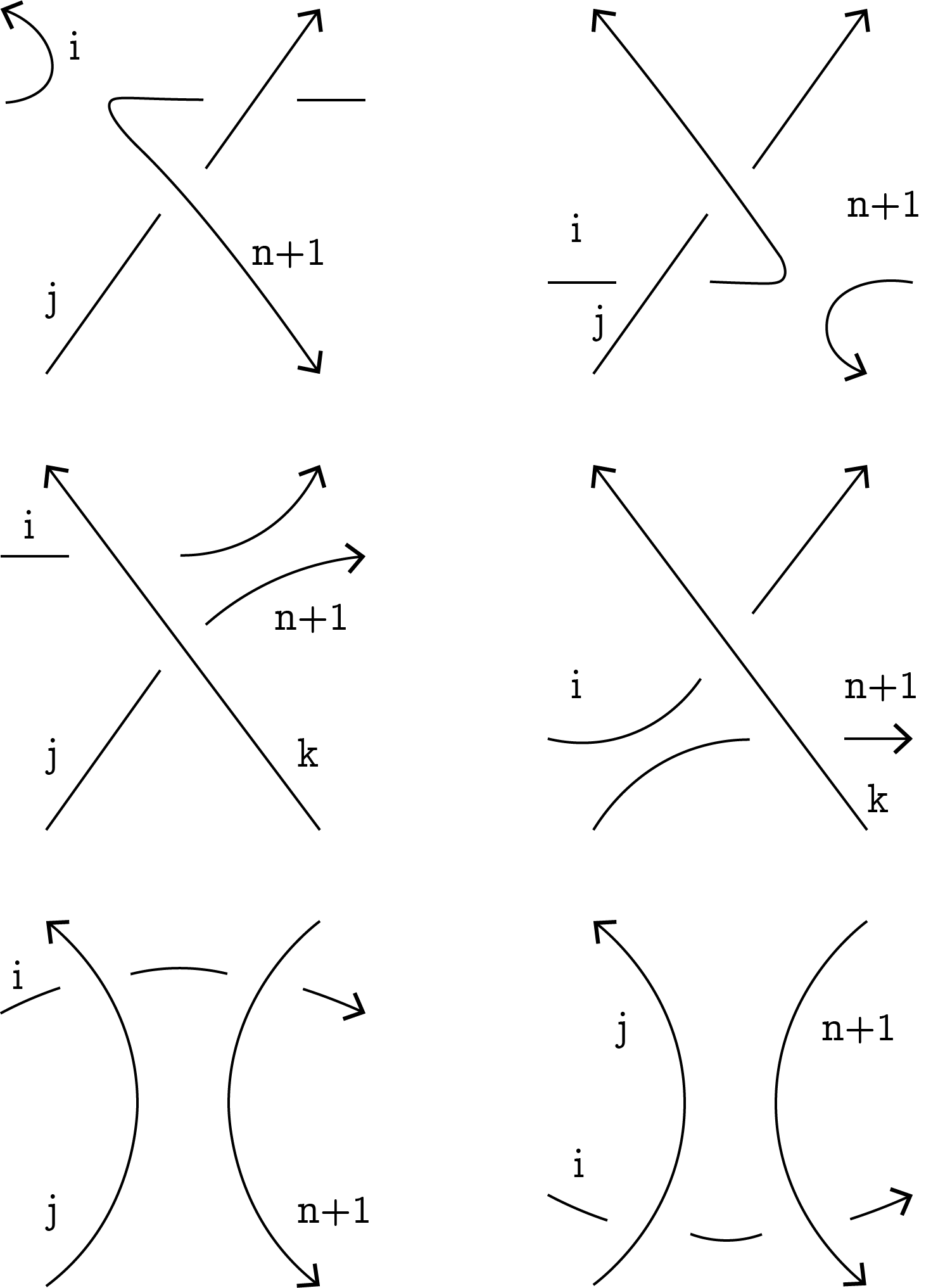}
\caption{Smoothings before and after RIII move.} \label{fig15}
\end{figure}
				
\underline{SV-move:} Consider SV-move presented in Figure~\ref{fig16}, were link components are numerated by $i$, $j$, and $k$, where some of these numbers may coincide.  
\begin{figure}[ht]
\centering
\includegraphics[scale=0.5]{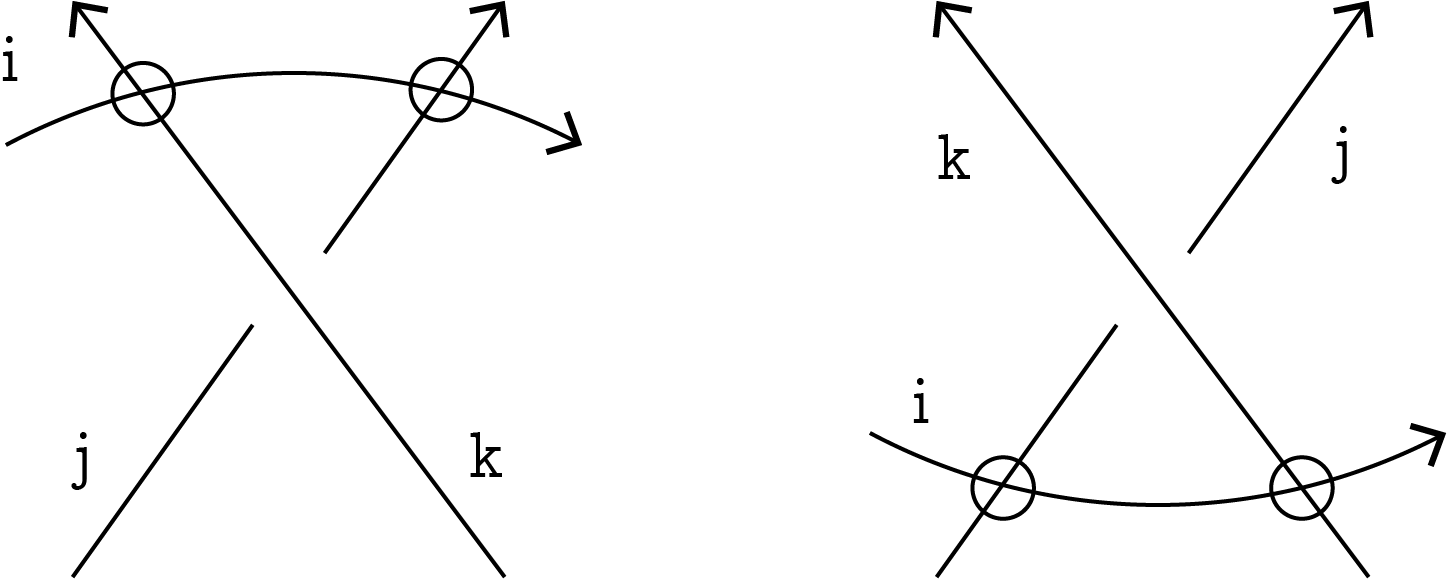}
\caption{SV-move.} \label{fig16}
\end{figure}

Let us apply type-3 smoothing in a crossing of the initial (before SV-move) diagram and  a crossing of the terminal (after SV-move) diagram, see Figure~\ref{fig17}. It is clear that smoothed diagrams are equivalent. 
\begin{figure}[ht]
\centering
\includegraphics[scale=0.5]{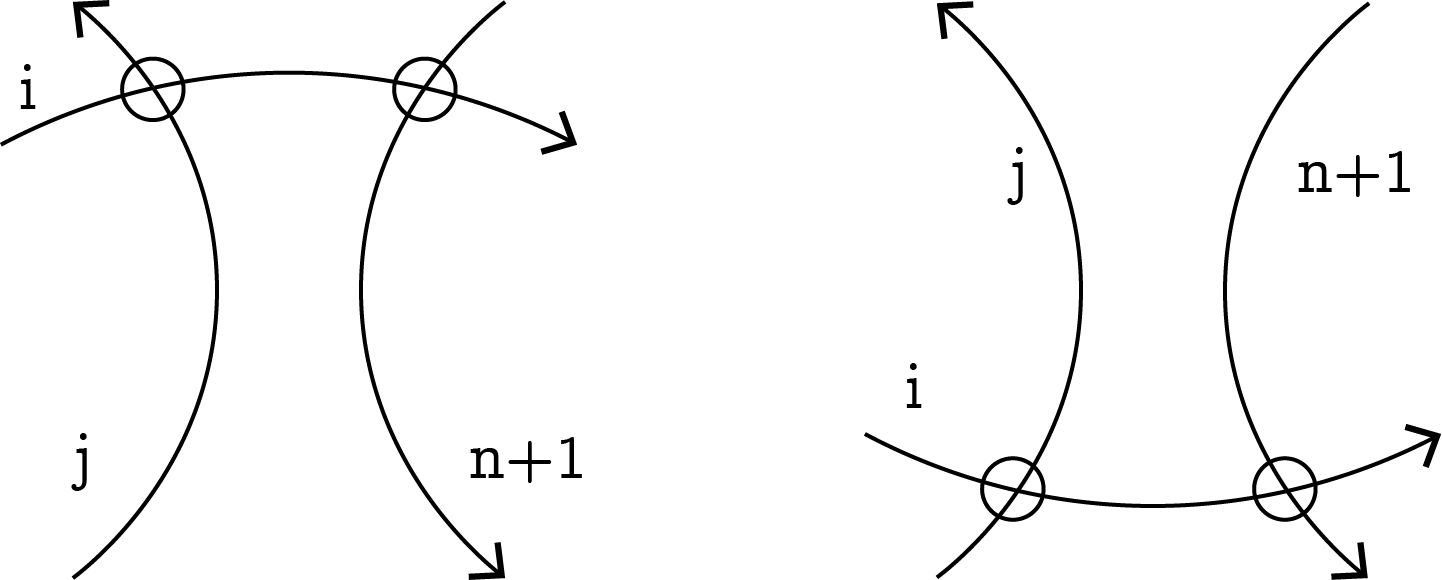}
\caption{Smoothing before and after SV-move.}
\label{fig17}
\end{figure}
Thus, we have shown that $S_2$ satisfies weight function conditions.

For $S_3$ the proof follows by analogous considerations. 
\end{proof}

\rm 

Since by Theorem~\ref{theorem2} function $S_2$ is an even weight function, according to Theorem~\ref{theorem1} we can construct the following invariants. 

\begin{corollary} \label{cor1} 
Consider function $B^i : D \to \VL$, defined by 
$$
B^i = \sum\limits_{c\in C_i(D)}\sgn(c) [K^c], 
$$
and flat virtual link invariant $B^i_{\text{flat}} : D \to \VL$, defined by 
$$
B^i_\text{flat} = \sum\limits_{c\in C_i(D)}\sgn(c) ([K^c]-[K^{c^*}]), 
$$
where $c \in  C_i(D)$ means that in the crossing $c$ both arcs belong to the $i$-th component of a link. Then $B^i$ is an ordered oriented virtual link invariant and $B^i_{\text{flat}}$ is an ordered oriented flat virtual link invariant. 
\end{corollary}

\rm

Invariants $B^i$ and  $B^i_{\text{flat}}$ appear to be useful for studying connected sums of virtual knots. We will demonstrate in Example~\ref{example5}  that these invariants can be used to prove that a Kishino knot, the famous connected sum of two trivial virtual knots~\cite[p.~23]{Fenn}, is non-trivial. 

\begin{example} \label{example5}
Let $K$ be an oriented virtual Kishino knot presented in Figure~\ref{fig18}. We will show that $B^1(K) \neq 0$, hence $K$ is distinguished  from  the unknot by $B^1$.
	 \begin{figure}[ht]
	 	\centering
	 	\includegraphics[scale=0.13]{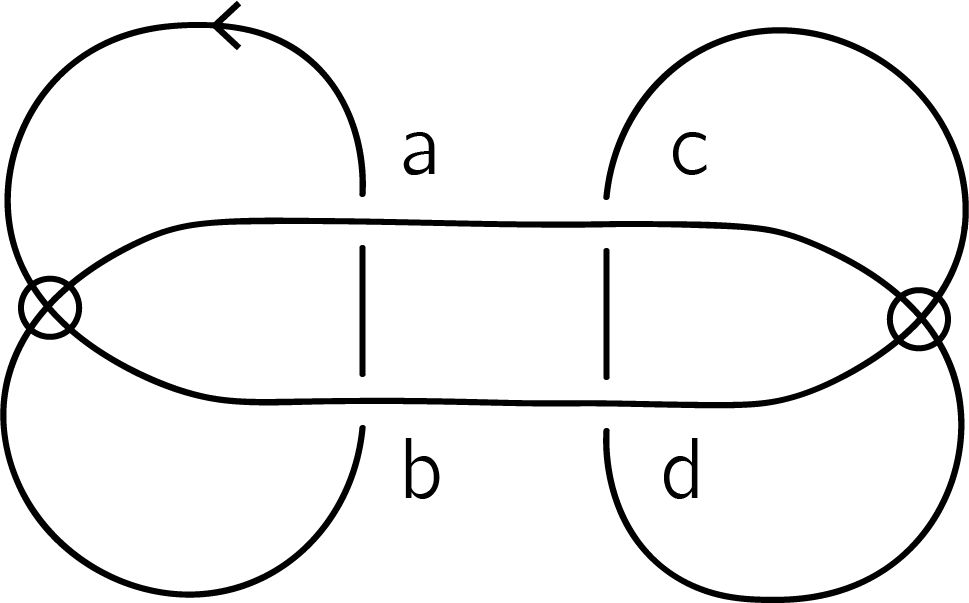}
	 	\caption{Kishino knot.} \label{fig18}
	 \end{figure}
	
\noindent	
 Denote the classical crossings by $a$,$b$,$c$, and $d$ as shown in Figure~\ref{fig18}. To calculate $B^1(K)$ we find signs of these crossings
 $$
\sgn(a) = \sgn(d) = -1 \quad and \quad \sgn(b) = \sgn(c) = 1. 
$$
Since $K$ is a knot, according to the definition of $B^1$, we consider smoothings in all crossings:  
$$
B^1(K) = -[K^a] + [K^b] + [K^c] - [K^d].
$$
Each smoothing provide an ordered oriented 2-component virtual link. One can check, that
$$
[K^a] = [K^d] \quad and \quad [K^b] = [K^c],
$$
so we only need to prove, that $[K^a]$ and $[K^b]$, presented in Figure~\ref{fig19}, are distinct.
\begin{figure}[ht]
\centering
\includegraphics[scale=0.13]{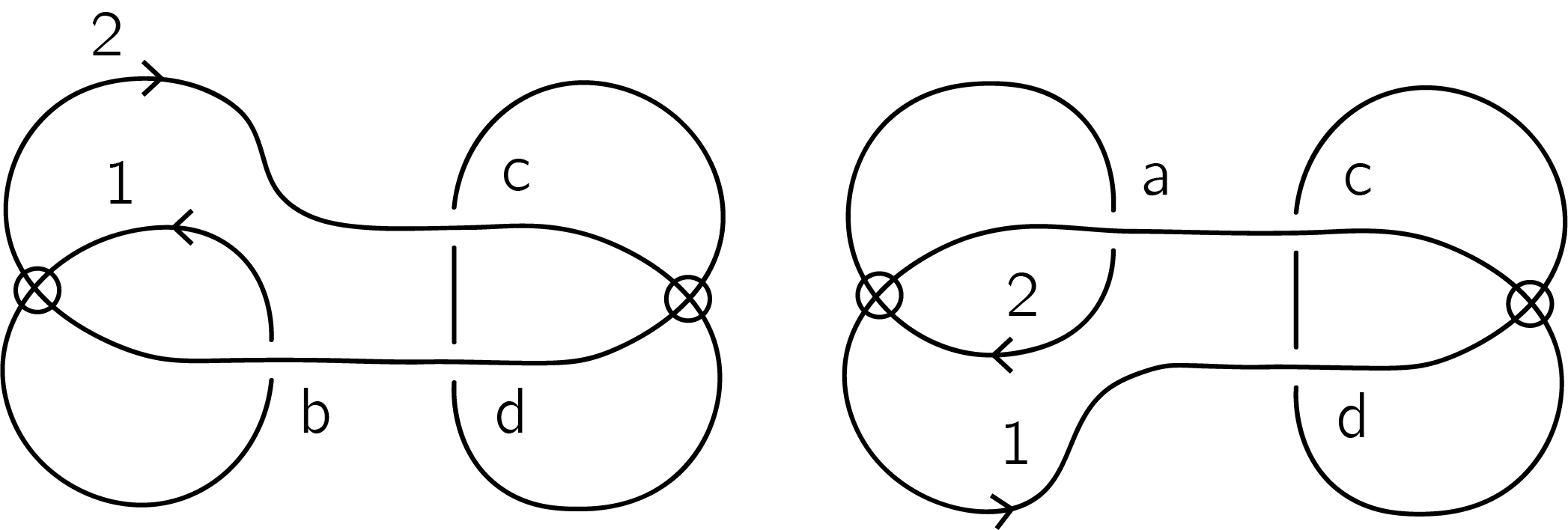}
\caption{Diagrams of $K_1 = K^a$ and $K_2 = K^b$.}
\label{fig19}
\end{figure}
To simplify notations we change notations, as $K_1 = K^a$ and $K_2 = K^b$. It is easy to see that only crossings $c$ and $d$ corresponds to $2$-nd component of $K_1$. Then
$$
B^2_\text{flat}(K_1) = [K_1^c] - [K_1^{c^*}] - [K_1^d] + [K_1^{d^*}]  \quad and \quad B^2_\text{flat}(K_2) = 0,
$$
where $c^*$ and $d^*$ are crossing changes of $c$ and $d$, respectively. Smoothed diagrams of $K_1$ are presented in Figure~\ref{fig20}, and smoothed diagrams, corresponding to $c^*$ and $d^*$, are presented in Figure~\ref{fig21}.  
	  	\begin{figure}[ht]
	  	\centering
	  	\includegraphics[scale=0.13]{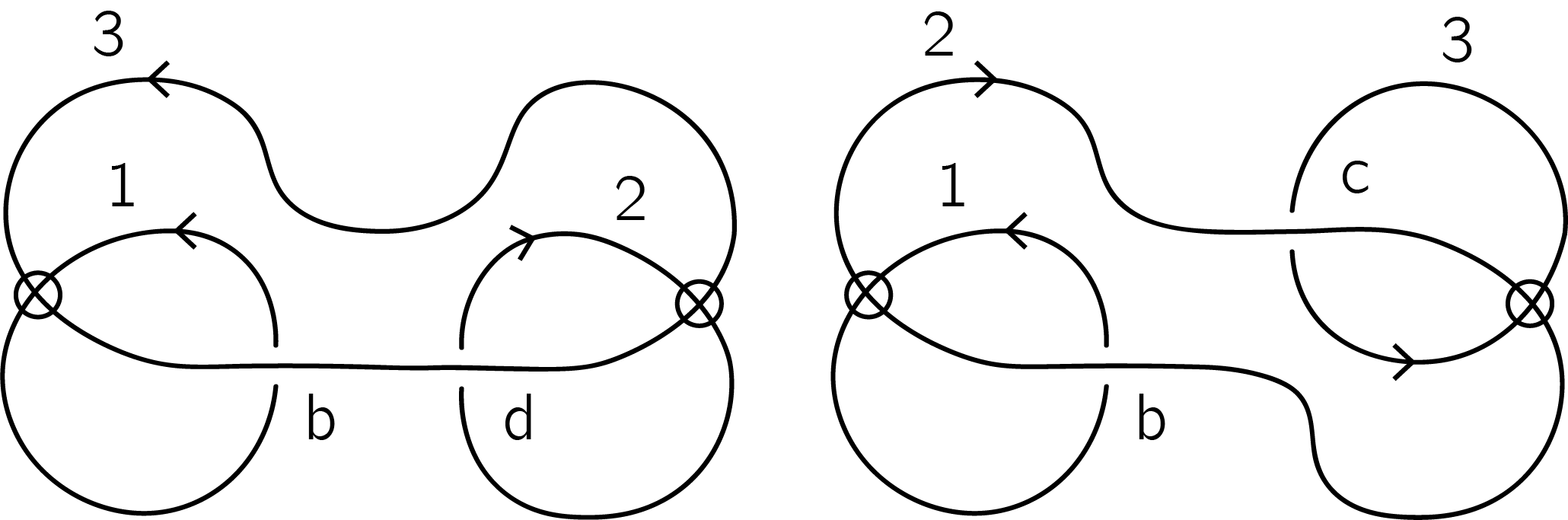}
	  	\caption{Diagrams of $K_1^c$ and $K_1^d$.} \label{fig20}
	  \end{figure}
	\begin{figure}[ht]
	  	\centering
	  	\includegraphics[scale=0.13]{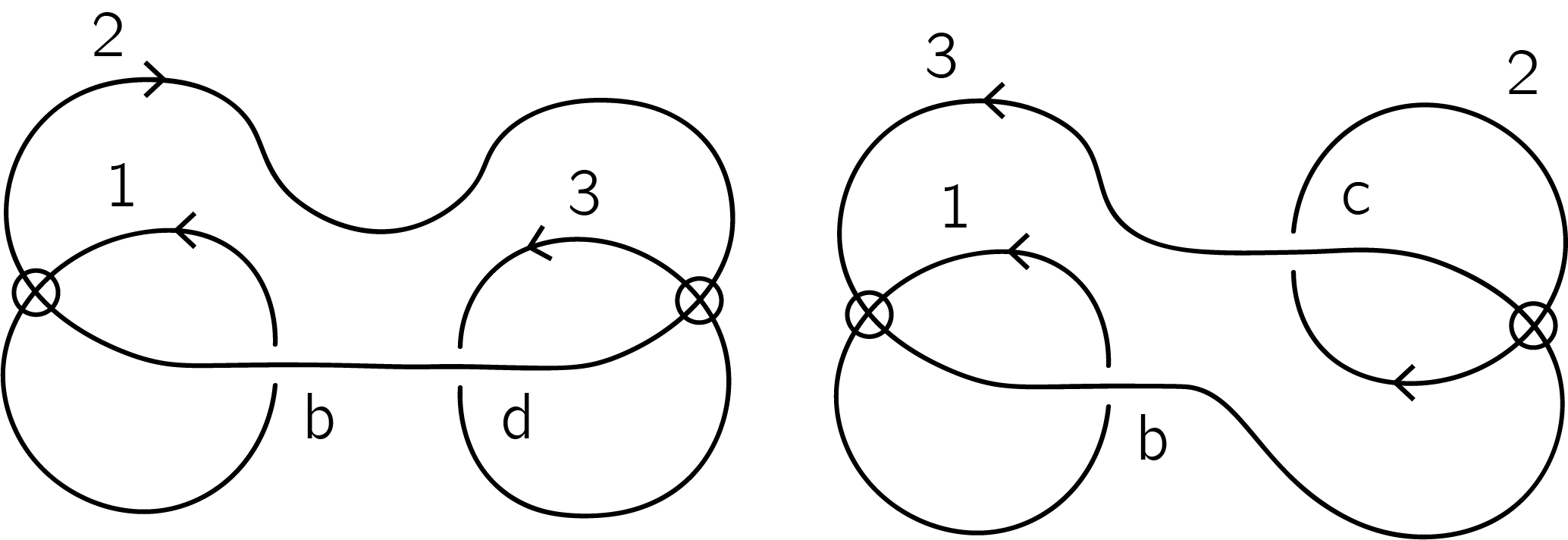}
	  	\caption{Diagrams of $K_1^{c^*}$ and $K_1^{d^*}$.} \label{fig21}
	  \end{figure}
In $B^2_\text{flat}(K_1)$ there are two summands with "+" and two with "-". Note that in $K_1^c$ and in $K_1^{d^*}$ the $3$-rd component is nontrivially linked with two other components, but for $K_1^{c^*}$ and $K_1^{d}$ it is not true. Hence, there are no cancelations and  $B^2_\text{flat}(K_1) \neq 0$. Therefore, Kishino knot is not equivalent to the unknot.  
\end{example}

\section{Recurrent construction of a sequence of invariants} \label{sec6}

In this section we will describe a recurrent construction of invariants based on Theorem~\ref{theorem1} and using recurrently defined odd and even weight functions. Below we consider the case of ordered links, and analogous statements holds for the case of unordered links.    

As well as above, $\VL$ denotes a free $\mathbb{Z}$-module generated by ordered oriented flat virtual links. Any invariant $A_1$ of ordered oriented flat virtual links taking values in a group $G$  may be extended to a homomorphism $A_1 :\VL \to G$. Similarly, any even weight function $F_1 \in W^{even}_{\VL}$ defines a weight function $A_1 \circ F_1 \in W^{even}_{G}$.

Suppose that there is given an even weight function $w_1 \in W^{even}_{H_1}$ for some abelian group $H_1$. Consider an odd weight function $u_1$ such that $u_1 \in W^{odd}_{Z}$ if $G$ is a group, and $u_1 \in W^{odd}_{G}$ if $G$ is a ring. Then according to Definitions~\ref{def5} and~\ref{def6} we can define $I$-function (see formula (\ref{eqn4})) and  flat $I_f$-function (see formula (\ref{eqn5})) by using even weight $w_1$ and odd weight $v_1  = u_1 * (A_1 \circ F_1)$, where the product $*$ of weights is defined by formula (\ref{eqn3}): 
$$
I (D; v_1, w_1, h_1) = \sum_{w_1(c) = h_1} v_1 (c)
$$
and
$$  
I_f (D; v_1, w_1, h_1) = \sum_{w_1(c) = h_1} v_1 (c) + \sum_{w_1(c^*) = h_1} v_1 (c^*),  
$$ 
for $h_1 \in H_1$ such that $h_1 \not\in R(w_1, L)$ or $R(v_1, L) = \{0\}$, where $D$ is a diagram of an ordered oriented virtual link $L$.  

The procedure can be repeated by taking odd weights $u_i \in W^{odd}_G$, even weights $F_i \in W^{even}_{\VL}$,  and even weights $w_i \in W^{even}_{H_i}$ for abelian groups $H_i$, where  $i = 2, 3, \ldots$. By Theorem~1, we get a family of  ordered oriented flat virtual link invariants $A_i : \VL \to G$ defined by  
$$
A_i(D) = I_f(D; v_{i-1}, w_{i-1}, h_{i-1}), \qquad \textrm {for } \quad i > 1, 
$$
where $v_{i-1} = u_{i-1} * (A_{i-1} \circ F_{i-1})$. Continuing the process, for $i> 1$ we will get I-functions 
$$
I (D; v_i, w_i, h_i) = \sum_{w_i(c) = h_i} v_i (c)
$$
and flat I-functions
$$  
I_f (D; v_i, w_i, h_i) = \sum_{w_i(c) = h_i} v_i (c) + \sum_{w_i(c^*) = h_i} v_i (c^*).  
$$ 
 
Note, that in general, $h_i$ cannot be arbitrary elements of $H_i$, since
$$
I_f(D; u_{i} * (A_{i} \circ F_{i}), w_{i}, h_{i})
$$
is guaranteed to be an invariant only for $h_{i} \notin R(w_{i},D)$. Finding proper elements $h_i$ to make $A_{i+1} \circ F_{i+1}$ a well-defined weight function seems a difficult problem. To avoid this problem we require  $R(u_i, D) = \{0\}$ for all $D$.
Then we can take $h_i$ equals to any element of the group $H_i$ and we obtain the following proposition.

\begin{proposition}\label{prop1}
Assume that there is an ordered oriented flat virtual link invariant $A_1$ taking values in a group $G$ and a sequence $\{(F_i, w_i, h_i, u_i)\}_{i \in \mathbb{N}}$ such that $F_i \in W^{even}_{\VL} $, $w_i \in W^{even}_{H_i}$, $h_i \in H_i$, and $u_i \in W^{odd}_\mathbb{Z}$  (or, $u_i \in W^{odd}_\mathbb{G}$ if $G$ is a ring) such that $R(u_i, D) = \{0\}$ for all link diagrams. Then there are even weight functions  $A_{i} \circ F_{i}$ and corresponding sequences of invariants 
$$
I(D; u_{i} * (A_{i} \circ F_{i}), w_{i}, h_{i}) \quad \text{and} \quad A_{i+1} = I_f(D; u_{i} * (A_{i} \circ F_{i}), w_{i}, h_{i}).
$$
\end{proposition}

The sequence we obtain is useful when $(F_i, w_i, h_i, u_i) = (F_1, w_1, h_1, u_1)$ for all $i$.

\begin{corollary}\label{cor2} Given three weight functions $v \in W^{odd}_H$, $w \in  W^{even}_G$ and $F \in W^{even}_{\VL}$ such that $R(w, L) = \{0\}$ for all links $L$ and a sequence $\{g_i \in G\}_{i \in \mathbb{N}}$ there is an infinite sequence of weight functions $\{v_i\}_{i \in \mathbb{N}}$, generated by them.
\end{corollary}

\begin{proof}
Let $A_1(D) = I_f(D, v, w, g)$. By taking $F_i = F$, $w_i = w$, $u_i = \sgn * w$ and applying Proposition~\ref{prop1} we obtain the desired sequence.
\end{proof}

\begin{remark} \label{remark2} \rm
On each step we may substitute $\VL$ by a module generated by a regular set of flat links $\mathcal{D}_{\text{flat}}$ and consider weight functions and invariants that are defined  for those links. Proposition~\ref{prop1} and Corollary~\ref{cor2} will remain true with some obvious changes.
\end{remark}

Further we will restrict ourselves to the case where $F_i$ are smoothings and provide some examples, generalizing already known invariants.
Before we continue with examples let us define a polynomial invariants, associated to those sequences.
For simplicity suppose that $G = H_i = \mathbb{Z}$.

\begin{theorem} \label{theorem3}
Let $S = \{s_1, \dots, s_k\}$ be a finite set of weight functions where $s_i = A_{m_i} \circ F_{m_i} \,$ for some  $m_i \in \mathbb{N}$, $w \in W^{odd}_\mathbb{Z}$ and $v \in W^{even}_\mathbb{Z}$ such that $R(v, D) = \{0\}$ for all $D$. Then
	$$
	\begin{gathered}
	F(t,  \ell_1, \ldots, \ell_k) = \sum_{c \in C(D)} w(c) t^{v(c)}\ell_1^{s_1(c)} \cdots \ell_k^{s_k(c)} - \sum_{c \in T(D)} w(c) \ell_1^{s_1(c)} \cdots \ell_k^{s_k(c)}
	\\- \sum_{c \not\in T(D)} w(c)  \ell_1^{A_{m_1}(D)} \cdots \ell_k^{A_{m_k}(D)}
	\end{gathered}
	$$
is a link invariant, where $T(D) =\{c \in C(D) \ | \ s_i(c) \in R(s_i, L) \text{ for all } i\}$.
\end{theorem}

\begin{proof}
The proof is analogous to the proof of invariance of F-polynomials in~\cite[Theorem~3.3]{KPV18}.
\end{proof}

\begin{example} \label{example6} Let $v = \ind$, $w = \sgn$, $A_1 = \nabla J_n$ and $s_1 = \nabla J_n(D_c) = A_1 \circ F$, where F is a type-1 smoothing presented in Figure~\ref{fig5}. Then we get F-polynomials from~\cite{KPV18}: 
$$
F_{D}^{n}(t,\ell)  =  \sum_{c \in C(D)} w (c)  t^{v(c)} \ell^{s_1(c)} - \sum_{c\in T(D)}w(c) \ell^{s_1(c)}  - \sum_{c\notin T(D)} w(c) \ell^{A_1(D)},
$$
where  $T(D)$ is a set of crossings of $D$ having the following property:
$$
T_{n}(D)=\{c \in D  \, \mid  \,  \nabla J_{n}(D_{c}) = \pm \nabla J_{n}(D)\}.
$$
\end{example}

\section{(n,m)-difference writhe and invariants of virtual knots} \label{sec7}

Now we shall continue the recursive procedure and define new invariants generalizing $F$-polynomials.
Take $F_1$ to be a type-1 smooting, and let $F_i = F_1$ for $i\geq 2$. Recall that $\sgn \in W^{odd}_{\mathbb Z}$ and $\ind \in W^{even}_{\mathbb Z}$, and similar to Example~\ref{example4} (formular (\ref{eqn9})) define  $A_1 = I_{f}(D; \sgn, \ind, n) = \nabla J_n$. 
Denoting $ u_i = \sgn * \ind$ for $i \geq 1$ we define 
\begin{eqnarray*}
\nabla J_{n,m} (D) & = & A_2(D) =  I_{f}(D; u_2 * (A_1 \circ F_1), w_1, m) =  \\
&  = &   \sum_{\ind(c) = m} \sgn(c) \ind(c) \nJ_n(D_c) - \sum_{\ind(c) = -m} \sgn(c) \ind(c)\nJ_n(D_c) \\
 & = & m \sum_{\ind(c) = m} \sgn(c) \nJ_n(D_c).
\end{eqnarray*}
By Theorem~\ref{theorem1}, $\nabla J_{n,m} (D)$ is an oriented flat virtual knot invariant, and we call it the \emph{(n,m)-dwrithe} (difference writhe) of an oriented virtual knot $K$ whose diagram is $D$.

\begin{theorem} \label{theorem4}
(i) The polynomial
\begin{eqnarray*}
F^{n,m,k}_D (t,  \ell_1, \ell_2)  & = &  \sum_{c \in C(D)} \sgn(c) t^{\ind(c)} \ell_1^{\nabla J_n(D_c)} \ell_2 ^{\nabla J_{m,k}(D_c)} \\
& & - \sum_{c \in T(D)} \sgn(c) \ell_1^{\nabla J_n(D_c)}\ell_2^{\nabla J_{m,k}(D_c)} -  \sum_{c \not\in T(D)} \sgn(c) \ell_1^{\nabla J_n(D)}\ell_2 ^{\nabla J_{m,k}(D)} ,
\end{eqnarray*}
 where 
 $$
 T(D)=\{c \in C(D) \ | A_2 \circ F_2 \in R(A_2 \circ F_2, L) \text{ and }\ A_3 \circ F_3 \in R(A_3 \circ F_3, L)\}, 
 $$  
 is an oriented virtual  knot invariant. \\
 (ii) The family of  invariants $F_D^{n,m,k}(t, \ell_1, \ell_2)$ is stronger than invariants $F^n_D(t, \ell)$.
\end{theorem}

\begin{proof}
(i) By letting $S = \{A_2 \circ F_2 ,A_3 \circ F_3\}$, $w = \sgn \in W^{odd}_{\mathbb Z}$ and $v = \ind \in W^{even}_{\mathbb Z}$ in Theorem~\ref{theorem3} we obtain the desired polynomial invariant.

(ii) It follows from the definition that $F_D^{n,m,k}$ coincide with $F^n_D$ for large enough $k$. Thus it suffices to find a knot which cannot be distinguished from the unknot by $F$--polynomials, but can be distinguished by $F_D^{n,m,k}$--polynomials.

Firstly, let us consider an oriented  virtual knot $K = 4.31$ from the Green's table (see also~\cite{IV20}) as presented in Figure~\ref{fig22}.
The diagram has four classical crossings: $\alpha$, $\beta$, $\gamma$, and $\delta$.
Denote by $K_{\alpha}$ and $K_{\beta}$ oriented virtual knots obtained by type-1 smoothing at classical crossings $\alpha$ and $\beta$, respectively. Figure~\ref{fig22} presents labelings of arcs of diagrams of $K$, $K_{\alpha}$ and $K_{\beta}$ satisfying the rule from Figure~\ref{fig4}.
\begin{figure}[h]
\centering
\includegraphics[scale = 0.25]{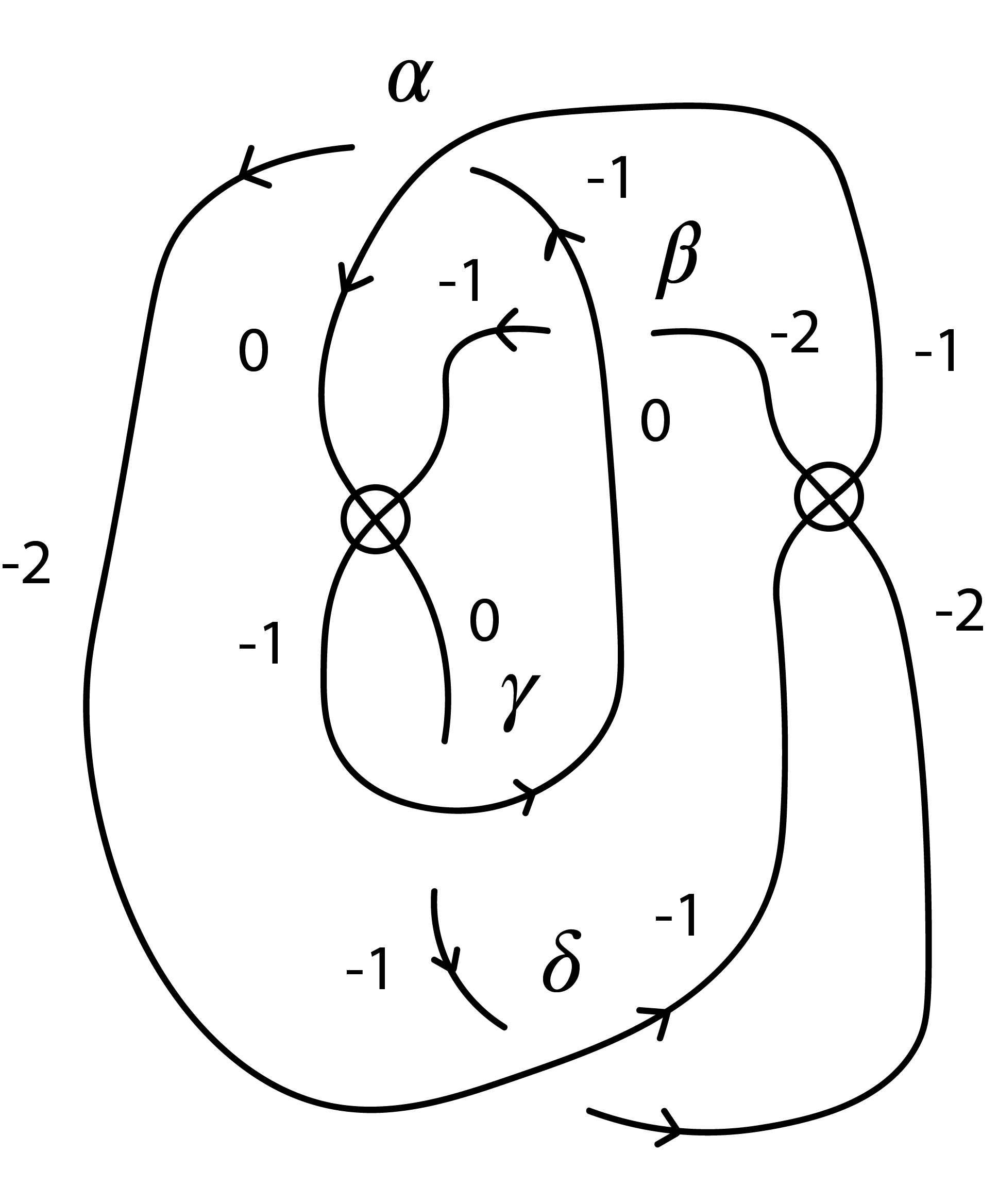} \quad
\includegraphics[scale = 0.25]{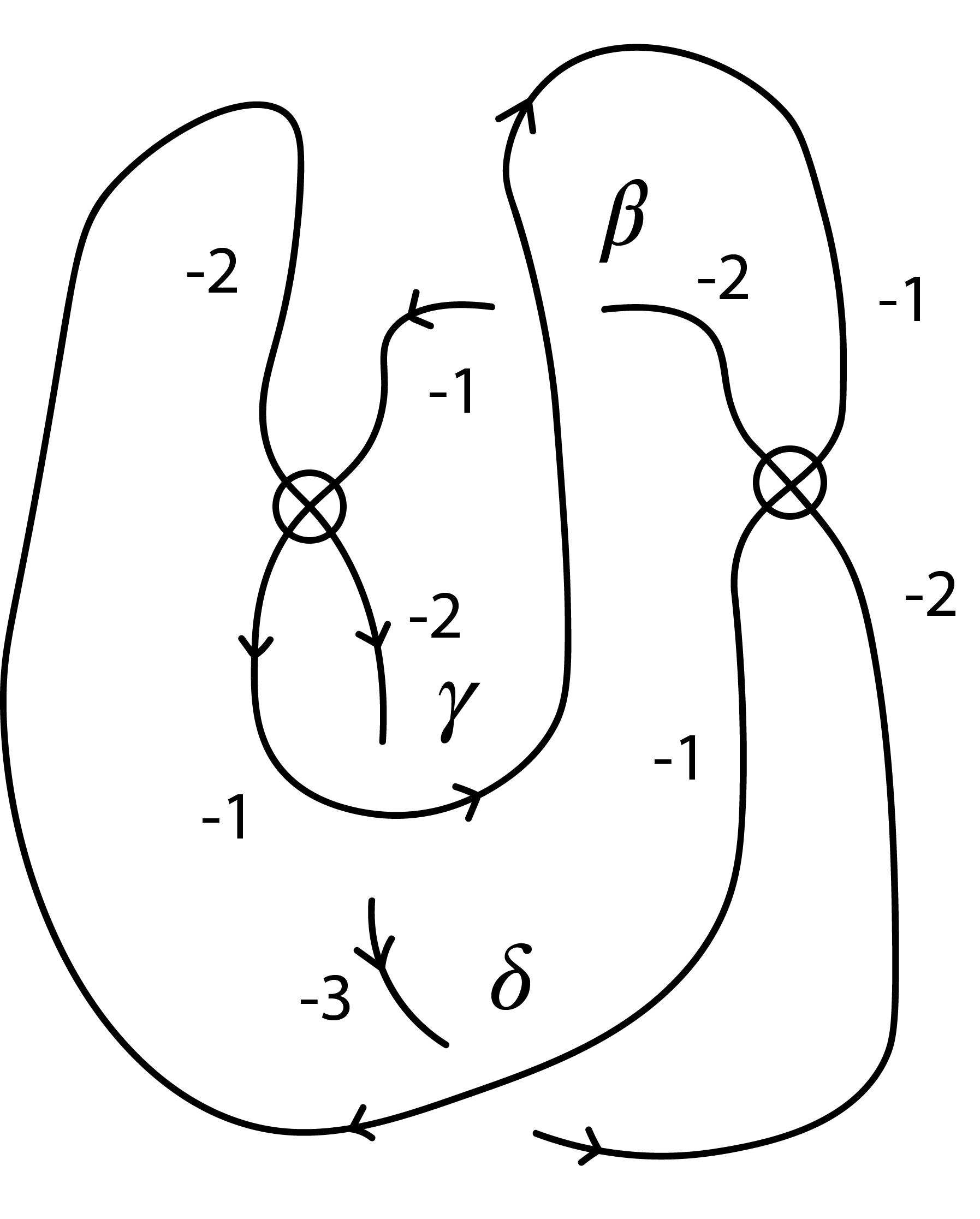} \quad
\includegraphics[scale =0.25]{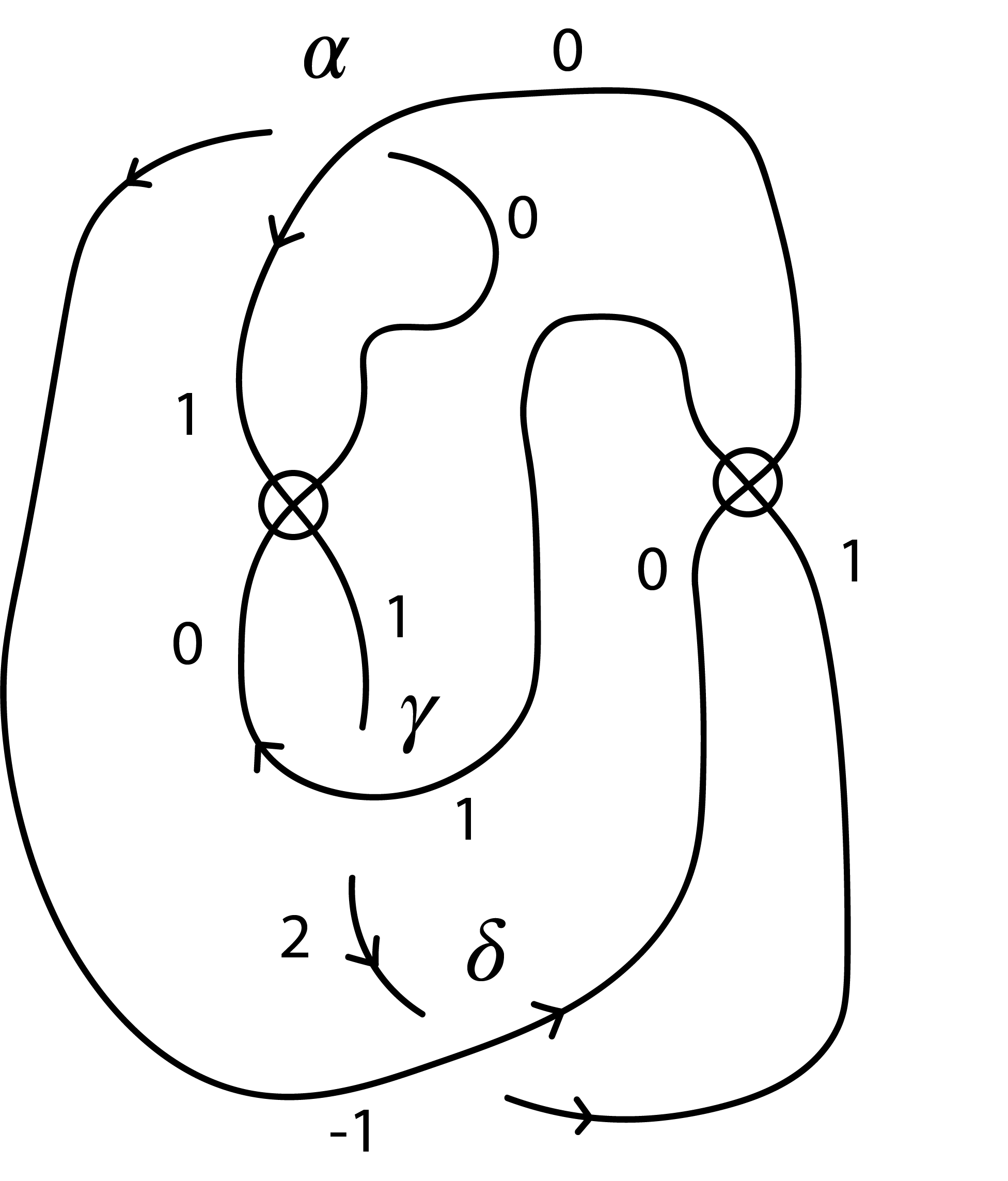}
\caption{Diagrams of knots $K$, $K_\alpha$ and $K_\beta$.} \label{fig22}
\end{figure}

Values of sign and index for classical crossings, and of $n$-dwrithe for knots $K$, $K_{\alpha}$ and $K_{\beta}$ are presented in Table~\ref{table1}.

\begin{table}[H] 
\caption{Values for $K$, $K_\alpha$ and $K_\beta$. \label{table1}}
\begin{tabular}{lll}
\hline
$K$  &$K_\alpha$  & $K_\beta$ \\  
\hline
$\sgn(\alpha) = -1$ &  $\sgn(\beta) = 1$ & $\sgn(\alpha) = -1$ \\
$\sgn(\beta) = 1$ & $\sgn(\gamma) = -1$ & $\sgn(\gamma) = 1$ \\
$\sgn(\gamma) = -1$  &  $\sgn(\delta) = 1$ & $\sgn(\delta) = -1$\\  
$\sgn(\delta) = -1$ & & \\ 
\hline
$\ind(\alpha) = 1$  & $\ind(\beta) = 1$ &  $\ind(\alpha) = 1$ \\
$\ind(\beta)  = 1$ & $\ind(\gamma) = 2$ &  $\ind(\gamma )= -1$ \\
$\ind(\gamma) = 0$  & $\ind(\delta)  = 1$ & $\ind(\delta) = -2$ \\ 
$\ind(\delta) = 0$ & & \\
\hline
$\nJ_1(K) = 0$ &$\nJ_1(K_{\alpha})  = 2$ & $\nJ_1 (K_{\beta})= -2$ \\
$ \nJ_2(K) = 0$ &  $\nJ_2(K_{\alpha}) = -1$ &   $ \nJ_2 (K_{\beta})= 1$  \\ 
\hline
\end{tabular}
\end{table}

To calculate the (1,1)-dwrithe of $K$ we observe that the index equals to $1$ for only two classical crossings: $\ind(\alpha) = 1$ and $\ind(\beta) = 1$. Therefore,
\begin{eqnarray*}
\nJ_{1,1}(K)  & = & 1 \cdot  \sum_{\ind(c)=1} \sgn(c) \nabla J_{1} (D_c)  \\
&  = &  \sgn(\alpha)  \nJ_1(K_\alpha) +  \sgn(\beta)  \nJ_1(K_\beta)  =   (-1) \cdot 2 + 1 \cdot (-2)  = -4.
\end{eqnarray*}

Now we consider an oriented virtual knot $K'$ with diagram presented in Figure~\ref{fig23}. 
\begin{figure}[!ht]
\centering
\includegraphics[scale = 0.26]{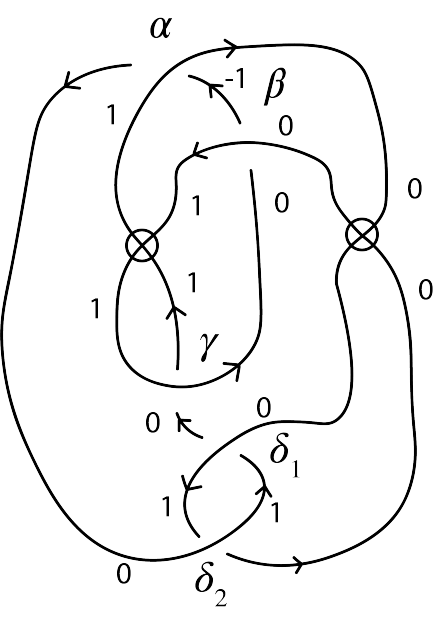}
\caption{Diagram of  $K'$.} \label{fig23}
\end{figure}
Virtual knot $K'$ has five classical crossings: $\alpha$, $\beta$, $\gamma$, $\delta_{1}$ and $\delta_{2}$. By direct calculations we get
$$
\begin{gathered}
\sgn(\beta) = \sgn(\delta_1) = \sgn(\delta_2) = -1, \qquad \sgn(\alpha) = \sgn(\gamma) = 1,\\
\ind(\gamma) = \ind(\delta_1) = \ind(\delta_2) = 0, \qquad  \ind(\alpha) = \ind(\beta) = 1.
\end{gathered}
$$
Therefore, $\nJ_1(K') = \nJ_2(K') = 0$.

It is easy to see, that $ K'_{\delta_1}$ is flat equivalent to the above considered $K=4.31$ and $ K'_{\delta_2}$ is flat equivalent to $K^{-}$. Therefore, $\nJ_{n} (K'_{\delta_{1}}) = \nJ_{2} (K'_{\delta_{1}}) = 0$ for $n=1,2$.
Moreover, all other smoothings, $K'_{\alpha}$, $K'_{\beta}$ and $K'_{\gamma}$, are flat equivalent to the unknot. Hence, we have  $\nJ_1(K'_c) = \nJ_2(K'_c) = 0$  for all classical crossings $c$ in the diagram.  Since $\nJ_1(K') = \nJ_2(K') = 0$, we get sets $T_{1} (K')= T_{2} (K')= \{ \alpha, \beta, \gamma, \delta_{1}, \delta_{2} \}$.   Therefore,
\begin{eqnarray*}
F_{K'}^{n}(t,\ell) & = & \sum_{c \in C(D)} \operatorname{sgn}(c) \left( t^{\text{Ind}(c)} - 1\right) \ell^{\nabla J_{n}(K_{c})}  \cr
& = & P_{K'} (t) = \sgn(\alpha) (t-1) + \sgn(\beta) (t-1) =0
\end{eqnarray*}
for any $n \in \mathbb N$. Thus, polynomials $F^{n}$ cannot distinguish $K'$ from the unknot.

At the same time, we have $\nJ_{1,1}(K') = 0$ by straightforward calculations, and $\nJ_{1,1}(K'_{\delta_1}) = -4$ and $\nJ_{1,1}(K'_{\delta_2}) = 4$ since $K'_{\delta_{1}}$ is flat equivalent to $K$ and $K'_{\delta_{2}}$ is flat equivalent to $K^{-}$. Therefore,
$T_{1,1} (D) = \{ \alpha, \beta, \gamma\}$ and
\begin{eqnarray*}
F^{1,1,1}_{K'}(t, \ell_1, \ell_2) & = & \sum_{c \in \{ \alpha, \beta, \gamma \}} \sgn (c) \left( t^{\ind(c)} - 1 \right) \ell_2^{\nJ_{1,1} (K'_{c})} \cr
& & \qquad
+  \sum_{c\in \{ \delta_{1}, \delta_{2} \}} \sgn(c) \left( t^{\ind(c)} \ell_2^{\nJ_{1,1} (K'_{c})} - \ell_2^{\nJ_{1,1} (K')} \right) \cr
& =  & (-1) (t^{0} \ell_2^{-4} -1) + (-1) (t^0 \ell_2^{4} -1) = -\ell_2^{-4} - \ell_2^{4} + 2.
\end{eqnarray*}
Thus, the polynomial $F^{1,1,1}$ does  distinguish virtual knot $K'$ from the unknot.
\end{proof}

\section{Flat span and invariants of 2-component virtual links} \label{sec8}

Let $L=K_1\cup K_2$ be an ordered oriented virtual 2-component link, where $K_{1}$ is the first component, and $K_{2}$ is the second component. Suppose that $L$ is presented by its diagram. Denote by $O(L)$ the set of crossings where $K_1$ passes over~$K_2$. Define the \emph{over linking number} $O_{\ell k}$ for $L$ as follows:
$$
O_{\ell k}(L) = \sum _{c \in O(L)} \operatorname{sgn}(c).
$$
Analogously, denote by $U(L)$ the set of crossings where $K_{1}$ passes under $K_{2}$ and define the \emph{under linking number} $U_{\ell k}$ for $L$ as follows:
$$
U_{\ell k}(L) = \sum _{c\in U(L)} \operatorname{sgn}(c).
$$

In virtual links the two linking numbers $ O_{\ell k}(L)$ and $ U_{\ell k}(L)$ may be not equal,  as can be seen for the ordered oriented virtual Hopf link $\mathcal H$ shown in Figure~\ref{fig24}. It is easy to see that $ O_{\ell k}(\mathcal H)= -1$ and $ U_{\ell k}(\mathcal H)= 0$. Also note that with reversing  order of components in $\mathcal H$, the two linking numbers will exchange. 
\begin{figure}[H]
\centering
\includegraphics[width=3.cm]{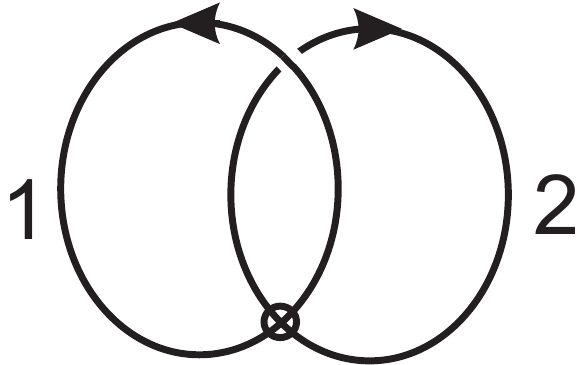}
\caption{Virtual Hopf link $\mathcal H$.}
  \label{fig24}
\end{figure}

\begin{definition}
For an ordered oriented virtual 2-component link $L$ define its \emph{span} by
$$
\operatorname{span} (L)= O_{\ell k}(L)- U_{\ell k}(L). %= r_+-r_--l_++l_-.
$$
\end{definition}

\begin{lemma} \label{lemma1}
$\operatorname{span}(L)$ is an invariant of ordered oriented 2-component virtual link $L$.
\end{lemma}

\begin{proof}
Since $O_{\ell k}(L)$ and $U_{\ell k}(L)$ are invariants, $\operatorname{span}(L)$ is an invariant as their difference.
\end{proof}

 \begin{lemma} \label{lemma2}
Let $L$ be an ordered oriented virtual link, and $L'$ is obtained from $L$ by reversing orientation on both component and preserving the order of its components. Then $\operatorname{span} (L')  = \operatorname{span} (L)$.
\end{lemma}

\begin{proof}
Changing orientation on both components does not change the signs of classical crossings, hence $\operatorname{span}(L)$ does not change.
\end{proof}

\begin{definition} \label{def8}
Let $D$ be  a diagram of an ordered oriented virtual 2-component link $L = K_{1} \cup K_{2}$, and $C_{12}(D)$ be the set of all classical crossings in $D$ in which $K_{1}$ and $K_{2}$ meet. For $c \in C_{12}(D)$ let $\prescript{c}{}{D}$ be a knot diagram obtained by type-3 smoothing at $c \in D$ (see Figure~\ref{fig10}). For $n, k \in \mathbb Z$ consider set 
$$
I_{n,k} = \{ c \in D : \nabla J_{n}(\prescript{c}{}{D})=k \}. 
$$  
Define \emph{(n,k)-span} for $L$ as follows:
$$
\operatorname{span}_{n, k}(L)=\sum\limits_{c \in O(L) \cap I_{n, k}} \operatorname{sgn}(c)-\sum\limits_{c \in U(L) \cap I_{n, k}}\operatorname{sgn}(c). 
$$
\end{definition}

\begin{definition} \label{def9}
For a 2-component link $L$ define \emph{flat span} of $L$ as follows:
$$
\operatorname{fspan}_{n,k}(L)= \operatorname{span}_{n,k}(L) + \operatorname{span}_{n,-k}(L).
$$
\end{definition}

\begin{theorem} \label{theorem5}
Let $D$ be a diagram of an ordered oriented virtual 2-component link. Then $\operatorname{span}_{n,k}(D)$ is a link invariant and  $\operatorname{fspan}_{n,k}(D)$ is a flat link invariant.
\end{theorem}

\begin{proof}
Let us consider two weight functions
$$
w_1(c) = \begin{cases}
\sgn(c), & c\in O(D), \\
-\sgn(c), & c \in U(D), \\
0, & \text{otherwise,}
\end{cases}
$$
and
$$
w_{2, n} (c) = \begin{cases} \nabla J_n(\prescript{c}{}{D}), & c \in O(D), \\
0, & \text{otherwise.} \end{cases}
$$

\noindent 
Then $w_{1} \in W^{odd}_\mathbb{Z}$, $w_{2,n} \in W^{even}_\mathbb{Z}$, and by Theorem~\ref{theorem1} (formular (\ref{eqn8})) the function 
\begin{eqnarray*}
I(D; w_1, w_{2,n}, k) & = &  \sum_{w_{2,n}(c) = k} w_{1} (c) \cr
& = & \sum_{c \in O(D) \cap I_{n,k}} \sgn(c) + \sum_{c \in U(D) \cap I_{n,k}} (-\sgn(c)) \cr
& = & \operatorname{span}_{n,k} (D) ,
\end{eqnarray*}
where  $I_{n,k} = \{ c \in D : w_{2,n} (c) = \nabla J_{n}(\prescript{c}{}{D})=k \}$
is an ordered oriented virtual link invariant.

Note, that $w_{2,n}^*(c) = w_{2,n}(c^*) = -w_{2,n}(c)$ since crossing change leads to change of orientation induced after smoothing, and $w_1^*(c) = w_1(c^*) = w_1(c)$ since $\sgn(c) = -\sgn(c^*)$. Moreover, $c \in O(D)$ if and only if $c^*\in U(D)$. Therefore, the value
\begin{eqnarray*}
I_{f}(K; w_1, w_{2,n}, k) & = & \sum_{w_{2,n}(c) = k} w_1(c) + \sum_{w_{2,n}^*(c) = k} w_1^*(c) \cr
& =  & \sum_{w_{2,n}(c) = k} w_1(c) + \sum_{w_{2,n}(c) = -k} w_1(c) \cr
& = &   \operatorname{span}_{n,k}(D) + \operatorname{span}_{n,-k}(D) \cr
& = &  \operatorname{fspan}_{n,k}(L),
 \end{eqnarray*}
is an ordered oriented flat virtual link invariant by Theorem~\ref{theorem1}.
\end{proof}

\begin{lemma} \label{lemma3}
Let $L= K_{1} \cup K_{2}$ be a 2-component link, and $L'  = K_{2} \cup K_{1}$ be obtained from $L$ by exchanging the order of its components. Then $\operatorname{fspan}_{n,k} (L') = -\operatorname{fspan}_{n,k} (L)$.
\end{lemma}

The following theorem gives a 3-parameter family of 3-variable polynomials which are invariants of oriented virtual knots.

\begin{theorem} \label{theorem6}
(i)  The polynomial
\begin{eqnarray*}
\widetilde F_{K}^{n, k, m}(t, \ell, v) & = & \sum_{c \in C(D)} \operatorname{sgn}(c) t^{\ind(c)} \ell^{\nabla J_{n}\left(D_{c}\right)} v^{\operatorname{fspan}_{k, m}\left(D^{c}\right)} \cr
& & - \sum_{c \in T(D)} \operatorname{sgn}(c) \ell^{\nabla J_{n}\left(D_{c}\right)} v^{\operatorname{fspan}_{k, m}\left(D^{c}\right)} 
   - \sum_{c \notin T(D)} \operatorname{sgn}(c) \ell^{\nabla J_{n}(D)} v^{\operatorname{fspan}_{k, m}\left(D^{c}\right)} ,  
\end{eqnarray*} 
where 
$$
T(D) = \{c \in C(D) \ | \ \nabla J_n(D_c) = \pm \nabla J_n(D) \text{ and } \operatorname{fspan}_{k,m}(D^c) = 0\}
$$
is an oriented virtual knot invariant. \\
(ii) The family of polynomial invariants $\widetilde F_D^{n,k,m}(t, \ell, v)$ is stron\-ger than polynomial invariants $F^n_D(t, \ell)$.
\end{theorem}

\begin{proof}
(i) Consider weight functions $s_1 = \nabla J_n \circ F_1$ and $s_2 = \operatorname{fspan}_{k,m} \circ F_2$, where $F_1$ is a type-1 smoothing and $F_2$ is a type-2 smoothing. Suppose $w = \sgn \in W^{odd}_{\mathbb Z}$ and $v = \ind \in W^{even}_{\mathbb Z}$. Then the statement follows from Theorem~\ref{theorem3}.  

(ii) It is easy to see from the definition of polynomials $\widetilde F_{K}^{n, k, m}(t, \ell, v)$ that substituting $v=1$ gives back polynomials $F_K^n(t, \ell)$. Therefore, if two virtual knots $K$ and $K'$ can be distinguished by $F_K^n(t, \ell)$ for some $n$, then $K$ and $K'$ can also be distinguished by $\widetilde F_{K}^{n, k, m}(t, \ell, v)$ for any values $k$ and $m$.

However in the case when $F_K^n(t, \ell)$ fails to distinguish virtual knots $K$ and $K'$, it may be possible to find a pair $(k,m)$ such that $K$ and $K'$ can be distinguished by  polynomials $\widetilde F_{K}^{n, k, m}(t, \ell, v)$.

Indeed, consider an infinite family of virtual knots $VK_{q}$, for positive integer $q$, presented in Figure~\ref{VKN}, where $c_{j}$, $j=1, \ldots, q$, are repeating blocks separated by dashed lines. These virtual knots were constructed in \cite{GPV20} to provide examples of  $q$-simplexes in the Gordian complex of virtual knots corresponding to the arc shift move defined in~\cite{GKP19}. It was shown in~\cite{GKP19} that the arc shift move together with generalized Reidemeister moves is an unknotting operation for virtual knots. By using the Kauffman $f$-polynomial it was shown in~\cite{GPV20} that virtual knots  $VK_{q}$ are distinct for distinct $q$.
\begin{figure}[h]
\unitlength=.25mm
\centering
\begin{picture}(0,90)(0,0)
\put(-300,0){\begin{picture}(0,0)
{\thinlines
\qbezier(90,-13)(90,-13)(90,-8)
\qbezier(90,-3)(90,-3)(90,2)
\qbezier(90,7)(90,7)(90,12)
\qbezier(90,17)(90,17)(90,22)
\qbezier(90,27)(90,27)(90,32)
\qbezier(90,37)(90,37)(90,42)
\qbezier(90,47)(90,47)(90,52)
\qbezier(90,57)(90,57)(90,62)
\qbezier(90,67)(90,67)(90,72)
\qbezier(90,77)(90,77)(90,82)
\qbezier(90,87)(90,87)(90,92)
}
{\thinlines
\qbezier(250,-13)(250,-13)(250,-8)
\qbezier(250,-3)(250,-3)(250,2)
\qbezier(250,7)(250,7)(250,12)
\qbezier(250,17)(250,17)(250,22)
\qbezier(250,27)(250,27)(250,32)
\qbezier(250,37)(250,37)(250,42)
\qbezier(250,47)(250,47)(250,52)
\qbezier(250,57)(250,57)(250,62)
\qbezier(250,67)(250,67)(250,72)
\qbezier(250,77)(250,77)(250,82)
\qbezier(250,87)(250,87)(250,92)
\put(45,90){\makebox(0,0)[cc]{$c_{0}$}}
\put(170,90){\makebox(0,0)[cc]{$c_{1}$}}
}
\thicklines
\qbezier(0,0)(0,0)(90,0)
\qbezier(0,0)(0,0)(0,30)
\qbezier(0,50)(0,50)(0,80)
\qbezier(0,80)(0,80)(90,80)
\qbezier(0,30)(0,30)(20,50)
\qbezier(0,50)(0,50)(6,44)
\qbezier(14,36)(20,30)(20,30)
\qbezier(20,30)(20,30)(40,50)
\qbezier(20,50)(20,50)(40,30)
\qbezier(40,30)(40,30)(60,50)
\qbezier(40,50)(40,50)(46,44)
\qbezier(54,36)(60,30)(60,30)
\qbezier(60,30)(60,30)(80,50)
\qbezier(60,50)(60,50)(80,30)
\qbezier(80,50)(80,50)(80,70)
\qbezier(80,70)(80,70)(90,70)
\qbezier(80,30)(80,30)(80,10)
\qbezier(80,10)(80,10)(90,10)
\put(30,40){\circle{6}}
\put(70,40){\circle{6}}
%%%
\qbezier(90,0)(90,0)(250,0)
\qbezier(90,10)(90,10)(120,10)
\qbezier(120,10)(120,10)(120,16)
\qbezier(120,24)(120,24)(120,26)
\qbezier(120,34)(120,34)(120,70)
\qbezier(120,70)(120,70)(90,70)
\qbezier(80,80)(80,80)(130,80)
\qbezier(130,80)(130,80)(130,34)
\qbezier(130,26)(130,26)(130,24)
\qbezier(130,16)(130,16)(130,10)
\qbezier(130,10)(130,10)(250,10)
\qbezier(100,20)(100,20)(100,60)
\qbezier(100,20)(100,20)(240,20)
\qbezier(240,20)(240,20)(240,70)
\qbezier(240,70)(240,70)(250,70)
\qbezier(100,20)(100,20)(100,60)
\qbezier(100,60)(100,60)(220,60)
\qbezier(110,30)(110,30)(110,50)
\qbezier(110,30)(110,30)(140,30)
\qbezier(110,50)(110,50)(140,50)
\put(120,50){\circle{6}}
\put(120,60){\circle{6}}
\put(130,50){\circle{6}}
\put(130,60){\circle{6}}
\qbezier(140,30)(140,30)(160,50)
\qbezier(140,50)(140,50)(160,30)
\qbezier(160,30)(160,30)(180,50)
\qbezier(160,50)(160,50)(166,44)
\qbezier(174,36)(180,30)(180,30)
\qbezier(180,30)(180,30)(200,50)
\qbezier(180,50)(180,50)(200,30)
\qbezier(200,30)(200,30)(220,50)
\qbezier(200,50)(200,50)(206,44)
\qbezier(214,36)(220,30)(220,30)
\put(150,40){\circle{6}}
\put(190,40){\circle{6}}
\qbezier(220,50)(220,50)(220,60)
\qbezier(220,30)(220,30)(230,30)
\qbezier(230,30)(230,30)(230,80)
\qbezier(230,80)(230,80)(250,80)
\end{picture}}
\put(-140,0){\begin{picture}(0,100)
\thicklines
{\thinlines
\qbezier(250,-13)(250,-13)(250,-8)
\qbezier(250,-3)(250,-3)(250,2)
\qbezier(250,7)(250,7)(250,12)
\qbezier(250,17)(250,17)(250,22)
\qbezier(250,27)(250,27)(250,32)
\qbezier(250,37)(250,37)(250,42)
\qbezier(250,47)(250,47)(250,52)
\qbezier(250,57)(250,57)(250,62)
\qbezier(250,67)(250,67)(250,72)
\qbezier(250,77)(250,77)(250,82)
\qbezier(250,87)(250,87)(250,92)
\put(170,90){\makebox(0,0)[cc]{$c_{2}$}}
}
\qbezier(90,0)(90,0)(250,0)
\qbezier(90,10)(90,10)(120,10)
\qbezier(120,10)(120,10)(120,16)
\qbezier(120,24)(120,24)(120,26)
\qbezier(120,34)(120,34)(120,70)
\qbezier(120,70)(120,70)(90,70)
\qbezier(80,80)(80,80)(130,80)
\qbezier(130,80)(130,80)(130,34)
\qbezier(130,26)(130,26)(130,24)
\qbezier(130,16)(130,16)(130,10)
\qbezier(130,10)(130,10)(250,10)
\qbezier(100,20)(100,20)(100,60)
\qbezier(100,20)(100,20)(240,20)
\qbezier(240,20)(240,20)(240,70)
\qbezier(240,70)(240,70)(250,70)
\qbezier(100,20)(100,20)(100,60)
\qbezier(100,60)(100,60)(220,60)
\qbezier(110,30)(110,30)(110,50)
\qbezier(110,30)(110,30)(140,30)
\qbezier(110,50)(110,50)(140,50)
\put(120,50){\circle{6}}
\put(120,60){\circle{6}}
\put(130,50){\circle{6}}
\put(130,60){\circle{6}}
\qbezier(140,30)(140,30)(160,50)
\qbezier(140,50)(140,50)(160,30)
\qbezier(160,30)(160,30)(180,50)
\qbezier(160,50)(160,50)(166,44)
\qbezier(174,36)(180,30)(180,30)
\qbezier(180,30)(180,30)(200,50)
\qbezier(180,50)(180,50)(200,30)
\qbezier(200,30)(200,30)(220,50)
\qbezier(200,50)(200,50)(206,44)
\qbezier(214,36)(220,30)(220,30)
\put(150,40){\circle{6}}
\put(190,40){\circle{6}}
\qbezier(220,50)(220,50)(220,60)
\qbezier(220,30)(220,30)(230,30)
\qbezier(230,30)(230,30)(230,80)
\qbezier(230,80)(230,80)(250,80)
\put(270,80){\makebox(0,0)[cc]{$\cdots$}}
\put(270,70){\makebox(0,0)[cc]{$\cdots$}}
\put(270,10){\makebox(0,0)[cc]{$\cdots$}}
\put(270,0){\makebox(0,0)[cc]{$\cdots$}}
\end{picture}}
\put(50,0){\begin{picture}(0,100)
{\thinlines
\qbezier(90,-13)(90,-13)(90,-8)
\qbezier(90,-3)(90,-3)(90,2)
\qbezier(90,7)(90,7)(90,12)
\qbezier(90,17)(90,17)(90,22)
\qbezier(90,27)(90,27)(90,32)
\qbezier(90,37)(90,37)(90,42)
\qbezier(90,47)(90,47)(90,52)
\qbezier(90,57)(90,57)(90,62)
\qbezier(90,67)(90,67)(90,72)
\qbezier(90,77)(90,77)(90,82)
\qbezier(90,87)(90,87)(90,92)
\put(170,90){\makebox(0,0)[cc]{$c_{q}$}}
}
\thicklines
\qbezier(90,0)(90,0)(250,0)
\qbezier(90,10)(90,10)(120,10)
\qbezier(120,10)(120,10)(120,16)
\qbezier(120,24)(120,24)(120,26)
\qbezier(120,34)(120,34)(120,70)
\qbezier(120,70)(120,70)(90,70)
\qbezier(90,80)(90,80)(130,80)
\qbezier(130,80)(130,80)(130,34)
\qbezier(130,26)(130,26)(130,24)
\qbezier(130,16)(130,16)(130,10)
\qbezier(130,10)(130,10)(250,10)
\qbezier(100,20)(100,20)(100,60)
\qbezier(100,20)(100,20)(240,20)
\qbezier(240,20)(240,20)(240,70)
\qbezier(240,70)(240,70)(250,70)
\qbezier(100,20)(100,20)(100,60)
\qbezier(100,60)(100,60)(220,60)
\qbezier(110,30)(110,30)(110,50)
\qbezier(110,30)(110,30)(140,30)
\qbezier(110,50)(110,50)(140,50)
\put(120,50){\circle{6}}
\put(120,60){\circle{6}}
\put(130,50){\circle{6}}
\put(130,60){\circle{6}}
\qbezier(140,30)(140,30)(160,50)
\qbezier(140,50)(140,50)(160,30)
\qbezier(160,30)(160,30)(180,50)
\qbezier(160,50)(160,50)(166,44)
\qbezier(174,36)(180,30)(180,30)
\qbezier(180,30)(180,30)(200,50)
\qbezier(180,50)(180,50)(200,30)
\qbezier(200,30)(200,30)(220,50)
\qbezier(200,50)(200,50)(206,44)
\qbezier(214,36)(220,30)(220,30)
\put(150,40){\circle{6}}
\put(190,40){\circle{6}}
\qbezier(220,50)(220,50)(220,60)
\qbezier(220,30)(220,30)(230,30)
\qbezier(230,30)(230,30)(230,80)
\qbezier(230,80)(230,80)(250,80)
\qbezier(250,10)(250,10)(250,70)
\qbezier(250,0)(250,0)(260,0)
\qbezier(250,80)(250,80)(260,80)
\qbezier(260,0)(260,0)(260,80)
\end{picture}}
\end{picture}
\caption{Virtual knot $VK_{q}$.} \label{VKN}
\end{figure}
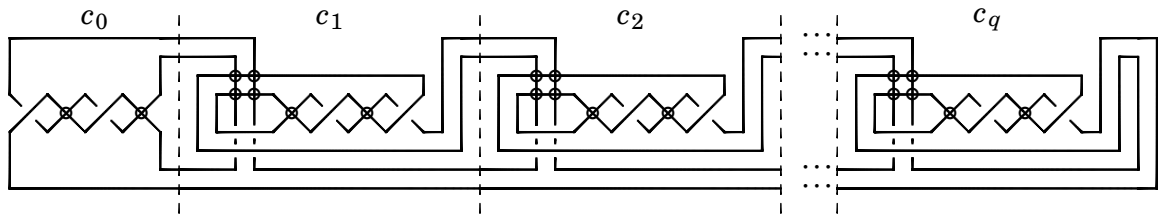

For $q \leqslant 10$ polynomials  $F_{VK_{q}}^n(t, \ell)$ were computed in~\cite{GPV20}.  It turns out that each of polynomials $F^{2}(t,\ell)$, $F^{4}(t,\ell)$ and $F^{6}(t,\ell)$ is able to distinguish virtual knots $VK_{1}$, $VK_{2}$ and $VK_{3}$, but not able to distinguish virtual knots $VK_{3}, \ldots, VK_{10}$.

It is shown in Table~\ref{tbl2}  that virtual knots $VK_{3}$ and $VK_{4}$  can be distinguished by each of polynomials $\widetilde F^{2, k, m}(t, \ell, v)$ for $(k,m) = (2, 0)$ and  $ (k,m)= (2, 2)$.
\begin{table}[H]
\caption{Polynomials $\widetilde F_{VK_{3}}^{2, k, m}(t, \ell,v)$ and $\widetilde F_{VK_{4}}^{2, k, m}(t, \ell,v)$.  \label{tbl2}}
\begin{tabular}{ccc}
\hline 
knot & $(k,m)$  & $\widetilde F^{2,k,m}(t, \ell, v)$  \\ 
\hline 
$VK_3$ & $(2,0)$ & $\begin{gathered} \ell^{-1}-t^{-2}v^{-8}-\ell^2v^6+\ell^{-2}+3\ell^2 +v^4+2v^6 +v^{-8} \cr -v^8-t^2v^4\ell^{-2}-t^2v^6\ell^{-4}+t^2v^8\ell^{-3}-5  \end{gathered} $ \\ 
\hline  
$VK_3 $ & $(2,2)$ & $ \begin{gathered} \ell^{-1}-v^2t^{-2}-\ell^2v^{-2}+\ell^{-2}+3\ell^2+v^{-2} \cr +v^2-t^2\ell^{-2}+t^2\ell^{-3}v^{-2}-t^2\ell^{-4}v^{-2}-4 \end{gathered} $ \\ 
\hline 
$VK_4$ & $(2,0)$ & $\begin{gathered} \ell^{-1}-t^{-2}v^{-8}-\ell^2v^6+\ell^{-2}+3\ell^2+v^4+2v^6+v^{-8} \cr - v^8-2t^2v^4\ell^{-2}+t^2v^4\ell^{-3}+t^2v^8\ell^{-2}-t^2v^6\ell^{-4}-5 \end{gathered}$ \\ 
\hline 
$VK_4$ & $(2,2)$ & $ \begin{gathered} \ell^{-1}-v^2t^{-2}-\ell^2v^{-2}+\ell^{-2}+3\ell^2+v^{-2}  \cr +v^2-2t^2\ell^{-2}+t^2\ell^{-3}+t^2\ell^{-2}v^{-2}-t^2\ell^{-4}v^{-2}-4 \end{gathered}$  \\ 
\hline
\end{tabular}
\end{table}
This observation completes the proof.
\end{proof}

%\vspace{6pt}

%%%%%%%%%%%%%%%%%%%%%%%%%%%%%%%%%%%%%%%%%%
%\subsection*{Acknowledgments} A. Gill and M. Prabhakar were supported by DST (project number DST/INT/RUS/RSF/P-02), M. Ivanov was supported by RFBR (grant number 19-01-00569),  A. Vesnin was  supported by RSF (grant number 20-61-46005).

%%%%%%%%%%%%%%%%%%%%%%%%%%%%%%%%%%%%%%%%%%

\newcommand{\etalchar}[1]{$^{#1}$}
%\begin{thebibliography}{{Smi}16a}

%%%%%%%%%%%%%%%%%%%%%%%%%%%%

\begin{thebibliography}{999}

\bibitem{BC20}%1
Bardakov V.,  Chuzhinov, B., Emel'yanenkov, I., Ivanov, M., Markhinina, E., Nasybulov, T., Panov, S., Singh, N., Vasyutkin, S., Yakhin. V.;   Vesnin, A.,  \emph{Representations of flat virtual braids which do not preserve the forbidden relations},. J. Knot Theory Ramifications accepted, preprint version is available at arXiv:2010.03162 (2020).

\bibitem{C16}%2
Cheng, Z., \emph{The Chord Index, its Definitions, Applications, and Generalizations}, Canadian Journal of Mathematics {\bf 73(3)} (2021), 597--621.

\bibitem{CG13}%3
Cheng, Z., Gao, H.,  \emph{A polynomial invariant of virtual links}, J. Knot Theory Ramifications {\bf 22(12)} (2013).

\bibitem{Fenn}%4
Fenn, R., Ilyutko, D.,  Kauffman, L., Manturov, V., \emph{Unsolved problems in virtual knot theory and combinatorial knot theory}, Banach Center Publications  {\bf 103} (2014), 9--61. 

\bibitem{GKP19}%5
Gill, A., Kaur, K., Prabhakar, M.,  \emph{Arc shift number and region arc shift number for virtual knots}, J. Korean Math. Soc. {\bf 56}   (2019),  1063--1081.

\bibitem{GPV20}%6
Gill, A., Prabhakar, M., Vesnin, A., \emph{Gordian complexes of knots and virtual knots given by region crossing changes and arc shift moves},  J.  Knot Theory Ramifications {\bf 29} (2020),  paper number 2042008.

\bibitem{Satoh21}%7
Higa, R., Nakamura, T.,  Nakanishi, Y., Satoh, S.,  \emph{The intersection polynomials of a virtual knot}, preprint arxiv:2102.12067.

\bibitem{IV20}%8
Ivanov, M., Vesnin, A.,  \emph{F-polynomials of tabulated virtual knots},  J. Knot Theory Ramifications  {\bf 29} (2020),  paper number 2050054.

\bibitem{Ka99}%9
Kauffman, L.H., \emph{Virtual knot theory}, European J. Combin.  {\bf 20} (1999),  663--690.

\bibitem{Ka13}%10
Kauffman, L.H.,  \emph{An affine index polynomial of virtual knots},  J. Knot Theory Ramifications {\bf 22}  (2013),  paper number 1340007.

\bibitem{Ka18}%11
Kauffman, L.H., \emph{Virtual knot cobordism and the affine index polynomial},  J. Knot Theory Ramifications  {\bf 27} (2018),  paper number  1843017.

\bibitem{Ka20}%12
Kauffman, L.H., \emph{The affine index polynomial and the Sawollek polynomial}, Preprint arxiv:2012.01738.

\bibitem{KPV18}%13
Kaur, K., Prabhakar, M.,  Vesnin, A., \emph{Two-variable polynomial invariants of virtual knots arising from flat virtual knot invariants},  J. Knot Theory Ramifications  {\bf 27} (2018),  paper number 1842015.

\bibitem{Ku03}%14
Kuperberg G., \emph{What is a virtual knot?}  Algebr. Geom. Topol.  {\bf 3} (2003), 587--591.

\bibitem{Me16}%15
Mellor, B.,  \emph{Alexander and writhe polynomials for virtual knots},  J. Knot Theory Ramifications  {\bf 27} (2016), paper number 1650050. 

\bibitem{Pe20}%16
Petit, N.,  \emph{The multi-variable affine index polynomial}, Topology and its Applications {\bf 274}  (2020), paper number 107145.

\bibitem{Sa16}%17 
Sakurai, M.,  \emph{An affine index polynomial and the forbidden move of virtual knots}, J. Knot Theory Ramifications {\bf 25} (2016),  paper number 1650040.

\bibitem{ST14}%18
Satoh, S., Taniguchi, K., \emph{The writhes of a virtual knot},  Fundamenta Mathematicae {\bf 225} (2014) 327--341.

\bibitem{VI20}%19 
Vesnin, A., Ivanov, M., \emph{ The polynomials of prime virtual knots of genus 1 and complexity at most 5},  Siberian Math. J.  {\bf 61} (2020),  994--1001.

\end{thebibliography}
\end{document}